\documentclass[final]{siamltex}
\input epsf

\usepackage{amsmath,amssymb,pstricks,color}
\usepackage{algorithm,algorithmic,xspace}

\newcommand{\nnum}{\nonumber}

\newcommand{\Ac}{\mathcal{A}}
\newcommand{\CC}{\mathcal{C}}

\newcommand{\EE}{\mathcal{E}}
\newcommand{\FF}{\mathcal{F}}
\newcommand{\GG}{\mathcal{G}}

\newcommand{\QQ}{\mathcal{Q}}
\newcommand{\LL}{\mathcal{L}}
\newcommand{\NN}{\mathcal{N}}

\newcommand{\until}[1]{\{1,\dots, #1\}}
\newcommand{\subscr}[2]{#1_{\textup{#2}}}
\newcommand{\supscr}[2]{#1^{\textup{#2}}}
\newcommand{\setdef}[2]{\{#1 \; | \; #2\}}

\newcommand{\FL}{\operatorname{FL}}

\newcommand{\inhomosyn}{\operatorname{DISCL}}
\newcommand{\homosyn}{\operatorname{DHSCL}}
\newcommand{\inhomoasyn}{\operatorname{DIACL}}
\newcommand{\homoasyn}{\operatorname{DHACL}}

\newtheorem{remark}{\bf Remark}[section]

\newtheorem{claim}{Claim}

\newcommand\oprocendsymbol{\hbox{$\bullet$}}
\newcommand\oprocend{\relax\ifmmode\else\unskip\hfill\fi\oprocendsymbol}

\begin{document}

\title{Distributed coverage games for mobile visual sensor networks}


\author{ Minghui Zhu and Sonia Mart{\'\i}nez\thanks{The authors are
    with Department of Mechanical and Aerospace Engineering,
    University of California, San Diego, 9500 Gilman Dr, La Jolla CA,
    92093, {\tt\small \{mizhu,soniamd\}@ucsd.edu}}}

\maketitle

\begin{abstract}
  Motivated by current challenges in data-intensive sensor
  networks, we formulate a coverage optimization problem for mobile visual sensors as a
  (constrained) repeated multi-player game. Each visual sensor tries to
  optimize its own coverage while minimizing the processing cost. We
  present two distributed learning algorithms where each sensor only
  remembers its own utility values and actions played during the last
  plays. These algorithms are proven to be convergent in probability
  to the set of (constrained) Nash equilibria and global optima of
  certain coverage performance metric, respectively.
\end{abstract}

\section{Introduction}

There is a widespread belief that continuous and pervasive monitoring
will be possible in the near future with large numbers of networked,
mobile, and wireless sensors.  Thus, we are witnessing an intense
research activity that focuses on the design of efficient control
mechanisms for these systems. In particular, decentralized algorithms
would allow sensor networks to react autonomously to changes in the
environment with minimal human supervision.

A substantial body of research on sensor networks has concentrated on
simple sensors that can collect scalar data; e.g.,~temperature,
humidity or pressure data. Here, a main objective is the design of
algorithms that can lead to optimal collective sensing through
efficient motion control and communication schemes. However, scalar
measurements can be insufficient in many situations; e.g.,~in
automated surveillance or traffic monitoring applications. In
contrast, data-intensive sensors such as cameras can collect visual
data that are rich in information, thus having tremendous potential
for monitoring applications, but at the cost of a higher processing
overhead.

Precisely, this paper aims to solve a coverage optimization problem
taking into account part of the sensing/processing trade-off.
Coverage optimization problems have mainly been formulated as
cooperative problems where each sensor benefits from sensing the
environment as a member of a group. However, sensing may also require
expenditure; e.g.,~the energy consumed or the time spent by image
processing algorithms in visual networks. Because of this, we endow
each sensor with a utility function that quantifies this trade-off,
formulating a coverage problem as a variation of congestion games
in~\cite{RWR:73}.

{\em Literature review.} In broad terms, the problem studied here is
related to a bevy of sensor location and planning problems in the
Computational Geometry, Geometric Optimization, and Robotics
literature. For example, different variations on the (combinatorial)
Art Gallery problem include~\cite{JOR:87}\cite{TCS:92}\cite{JU:00}.
The objective here is how to find the optimum number of guards in a
non-convex environment so that each point is visible from at least one
guard. A related set of references for the deployment of mobile robots
with omnidirectional cameras
includes~\cite{AG-JC-FB:05k}\cite{AG-JC-FB:06r}.  Unlike the Art
Gallery classic algorithms, the latter papers assume that robots have
local knowledge of the environment and no recollection of the
past. Other related references on robot deployment in convex
environments include~\cite{JC-SM-FB:03p}\cite{KL-JC:09} for
anisotropic and circular footprints.

The paper~\cite{IFA-TM-KC:07} is an excellent survey on multimedia
sensor networks where the state of the art in algorithms, protocols,
and hardware is surveyed, and open research issues are discussed in
detail. As observed in~\cite{RC:05}, multimedia sensor networks
enhance traditional surveillance systems by enlarging, enhancing, and
enabling multi-resolution views. The investigation of coverage
problems for static visual sensor networks is conducted
in~\cite{KYC-KSL-EYL:07}\cite{EH-RL:06}.

Another set of relevant references to this paper comprise those on the
use of game-theoretic tools to (i) solve static target assignment
problems, and (ii) devise efficient and secure algorithms for
communication networks. In~\cite{JRM-AW:08}, the authors present a
game-theoretic analysis of a coverage optimization problem for static
sensor networks. This problem is equivalent to the weapon-target
assignment problem in~\cite{RAM:99} which is nondeterministic
polynomial-time complete. In general, the solution to assignment
problems is hard from a combinatorial optimization viewpoint.

Game Theory and Learning in Games are used to analyze a variety of
fundamental problems in; e.g.,~wireless communication networks and the
Internet. An incomplete list of references includes~\cite{TA-TB-SD:06}
on power control,~\cite{TR:05} on routing, and~\cite{AT-LA:08} on flow
control. However, there has been limited research on how to employ
Learning in Games to develop distributed algorithms for mobile sensor
networks.  One exception is the paper~\cite{JRM-GA-JSS:08} where the
authors establish a link between cooperative control problems (in
particular, consensus problems), and games (in particular, potential
games and weakly acyclic games).

{\em Statement of contributions.} The contributions of this paper
pertain to both coverage optimization problems and Learning in Games.
Compared with~\cite{AK-SM:08b} and~\cite{KL-JC:09}, this paper employs
a more accurate sensing model and the results can be easily extended
to include non-convex environments. Contrary to~\cite{AK-SM:08b}, we
do not consider energy expenditure from sensor motions.

Regarding Learning in Games, we extend the use of the payoff-based
learning dynamics in~\cite{JRM-JSS:08}\cite{JRM-HPY-GA-JSS:06}.  The
coverage game we consider here is shown to be a (constrained) potential
game. A number of learning rules; e.g.,~better (or best) reply dynamics
and adaptive play, have been proposed to reach Nash equilibria in
potential games.  In these algorithms, each player must have access to
the utility values induced by alternative actions. In our problem
set-up; however, this information is unaccessible because of the
information constraints caused by unknown rewards, motion and sensing
limitations.  To tackle this challenge, we develop two distributed
payoff-based learning algorithms where each sensor only remembers its
own utility values and actions played during the last plays.

In the first algorithm, at each time step, each sensor repeatedly
updates its action synchronously, either trying some new action or
selecting the action which corresponds to a higher utility value in
the most recent two time steps.  The first advantage of this algorithm
over the payoff-based learning algorithms
of~\cite{JRM-JSS:08}\cite{JRM-HPY-GA-JSS:06} is its simpler dynamics,
which reduces the computational complexity. Furthermore, the algorithm
employs a diminishing exploration rate (in contrast to the constant
one in~\cite{JRM-JSS:08}\cite{JRM-HPY-GA-JSS:06}). The dynamically
changing exploration rate renders the algorithm an inhomogeneous
Markov chain (instead of the homogeneous ones
in~\cite{JRM-JSS:08}\cite{JRM-HPY-GA-JSS:06}). This technical change
allows us to prove convergence in probability to the set of
(constrained) Nash equilibria from which no agent is willing to
unilaterally deviate. Thus, the property is substantially stronger
than those in~\cite{JRM-JSS:08}\cite{JRM-HPY-GA-JSS:06} where the
algorithms are guaranteed to converge to Nash equilibria with a
sufficiently large probability by choosing a sufficiently small
exploration rate in advance.

The second algorithm is asynchronous. At each time step, only one
sensor is active and updates its state by either trying some new
action or selecting an action according to a Gibbs-like distribution
from those played in last two time steps when it was active. The
algorithm is shown to be convergent in probability to the set of
global maxima of a coverage performance metric. Compared with the
synchronous payoff-based log-linear learning algorithm
in~\cite{JRM-JSS:08}, this algorithm is asynchronous and
simpler. Furthermore, rather than maximizing the associated potential
function, the second algorithm optimizes a different global function
which captures better a global trade-off between the overall network
benefit from sensing and the total energy the network consumes. Again,
by employing a diminishing exploration rate, our algorithm is
guaranteed to have stronger convergence properties that the ones
in~\cite{JRM-JSS:08}.


\section{Problem formulation}\label{sec:problem}

Here, we first review some basic game-theoretic concepts; see, for
example~\cite{DF-JT:91}. This will allow us to formulate subsequently
an optimal coverage problem for mobile visual sensor networks as a
repeated multi-player game.  We then introduce notation used
throughout the paper.

\subsection{Background in Game Theory}\label{sec:gametheory}

A strategic game $\Gamma := \langle V, A, U \rangle$ has three
components:
\begin{itemize}
\item[1.] A set $V$ enumerating players $i \in V := \{1,\cdots,N\}$.
\item[2.] An action set $A := \prod_{i=1}^N A_i$ is the space of all
  actions vectors, where $s_i\in A_i$ is the action of player $i$ and
  an (multi-player) action $s\in A$ has components $s_1,\dots, s_N$.
\item[3.] The collection of utility functions $U$, where the utility
  function $u_i : A \rightarrow \mathbb{R}$ models player $i$'s
  preferences over action profiles.
\end{itemize}

Denote by $s_{-i}$ the action profile of all players other than $i$,
and by $A_{-i}=\prod_{j\neq i}A_j$ the set of action profiles for all
players except $i$.  The concept of (pure) Nash equilibrium (NE, for
short) is the most important one in Non-cooperative Game
Theory~\cite{DF-JT:91} and is defined as follows.

\begin{definition}[Nash equilibrium~\cite{DF-JT:91}] Consider the
  strategic game $\Gamma$. An action profile $s^{*} := (s_i^*,
  s_{-i}^*)$ is a (pure) NE of the game $\Gamma$ if $\forall i\in V$
  and $\forall s_i\in A_i$, it holds that $u_i(s^*)\geq u_i(s_i,
  s_{-i}^*)$. \label{def3}
\end{definition}

An action profile corresponding to an NE represents a scenario where
no player has incentive to unilaterally deviate. Potential Games form
an important class of strategic games where the change in a player's
utility caused by a unilateral deviation can be measured by a
potential function.
\begin{definition}[Potential game~\cite{DM-LS:96}] The strategic game
  $\Gamma$ is a potential game with potential function $\phi :
  A\rightarrow {\mathbb{R}}$ if for every $i\in V$, for every
  $s_{-i}\in A_{-i}$, and for every $s_i, s_i'\in A_i$, it holds that
  \begin{align}
    \phi(s_i, s_{-i})-\phi(s_i', s_{-i}) = u_i(s_i,
    s_{-i})-u_i(s_i', s_{-i}).\label{e4}
  \end{align}
  \label{def2}
\end{definition}

In conventional Non-cooperative Game Theory, all the actions in
$A_{i}$ always can be selected by player $i$ in response to other
players' actions. However, in the context of motion coordination, the
actions available to player $i$ will often be constrained to a
state-dependent subset of $A_i$. In particular, we denote by $F_i(s_i,
s_{-i})\subseteq A_i$ the set of feasible actions of player $i$ when
the action profile is $s := (s_i, s_{-i})$. We assume that $F_i(s_i,
s_{-i})\neq \emptyset$. Denote $F(s) := \prod_{i\in V}F_i(s) \subseteq
A$, $\forall s\in A$ and $F := \cup\{F(s)\;|\;s\in A\}$. The
introduction of $F$ leads naturally to the notion of constrained
strategic game $\subscr{\Gamma}{res} := \langle V, A, U, F \rangle$,
and the following associated concepts.
\begin{definition}[Constrained Nash equilibrium] Consider the
  constrained strategic game $\subscr{\Gamma}{res}$.  An action profile
  $s^{*}$ is a constrained (pure) NE of the game $\subscr{\Gamma}{res}$
  if $\forall i\in V$ and $\forall s_i\in F_i(s_i^*, s_{-i}^*)$, it
  holds that $u_i(s^*)\geq u_i(s_i, s_{-i}^*)$. \label{def1}
\end{definition}

\begin{definition}[Constrained potential game] The game
  $\subscr{\Gamma}{res}$ is a constrained potential game with potential
  function $\phi(s)$ if for every $i\in V$, every $s_{-i}\in A_{-i}$,
  and every $s_i\in A_i$, the equality~\eqref{e4} holds for every
  $s_i'\in F_i(s_i, s_{-i})$. \label{def4}
\end{definition}

Observe that if $s^*$ is an NE of the strategic game $\Gamma$, then it
is also a constrained NE of the constrained strategic game
$\subscr{\Gamma}{res}$. For any given strategic game, NE may not
exist. However, the existence of NE in potential games is
guaranteed~\cite{DM-LS:96}.  Hence, any constrained potential game has
at least one constrained NE.

\subsection{Coverage problem formulation}
\subsubsection{Mission space}

We consider a convex 2-D mission space that is discretized into a
(squared) lattice.  We assume that each square of the lattice has unit
dimensions. Each square will be labeled with the coordinate of its
center $q=(q_x, q_y)$, where $q_x\in [q_{x_{\min}}, q_{x_{\max}}]$ and
$q_y\in [q_{y_{\min}}, q_{y_{\max}}]$, for some integers
$q_{x_{\min}},\; q_{y_{\min}},$ $ q_{x_{\max}},\; q_{y_{\max}}$.
Denote by ${\mathcal{Q}}$ the collection of all squares of the
lattice.

We now define an associated location graph $\subscr{\GG}{loc} :=
({\QQ}, \subscr{E}{loc})$ where $((q_x, q_y),$ $(q_{x'}, q_{y'}))\in
\subscr{E}{loc}$ if and only if $|q_x -q_{x'}|+|q_y - q_{y'}|=1$ for
$(q_x, q_y), (q_{x'}, q_{y'})\in{\mathcal{Q}}$. Note that the graph
$\subscr{\GG}{loc}$ is undirected; i.e., $(q, q')\in \subscr{E}{loc}$
if and only if $(q', q)\in \subscr{E}{loc}$. The set of neighbors of
$q$ in $\subscr{E}{loc}$ is given by $\supscr{\NN}{loc}_{q} :=
\{q'\in\QQ\setminus \{q\} \;|\; (q, q')\in \subscr{E}{loc}\}$. We
assume that the location graph ${\GG}_{\rm loc}$ is fixed and
connected, and denote its diameter by $D$.

Agents are deployed in $\QQ$ to detect certain events of interest. As
agents move in $\QQ$ and process measurements, they will assign a
numerical value $W_{q}\ge 0$ to the events in each square with center
$q \in \QQ$.  If $W_{q}=0$, then there is no significant event at the
square with center $q$. The larger the value of $W_{q}$ is, the more
interest the set of events at the square with center $q$ is of.
Later, the amount $W_{q}$ will be identified with a benefit of
observing the point $q$. In this set-up, we assume the values $W_{q}$
to be constant in time. Furthermore, $W_q$ is not a prior knowledge to
the agents, but the agents can measure this value through sensing the
point $q$.


\subsubsection{Modeling of the visual sensor nodes}

Each mobile agent~$i$ is modeled as a point mass in $\QQ$, with
location $a_i := (x_i, y_i) \in \QQ$. Each agent has mounted a pan-tilt-zoom camera,
and can adjust its orientation and focal length.

The visual sensing range of a camera is directional, limited-range, and has a
finite angle of view. Following a geometric simplification, we model
the visual sensing region of agent~$i$ as an annulus sector in the
2-D plane; see Figure~\ref{fig1}.

\begin{figure}[th]
  \centerline{ \epsfxsize=3in \epsffile{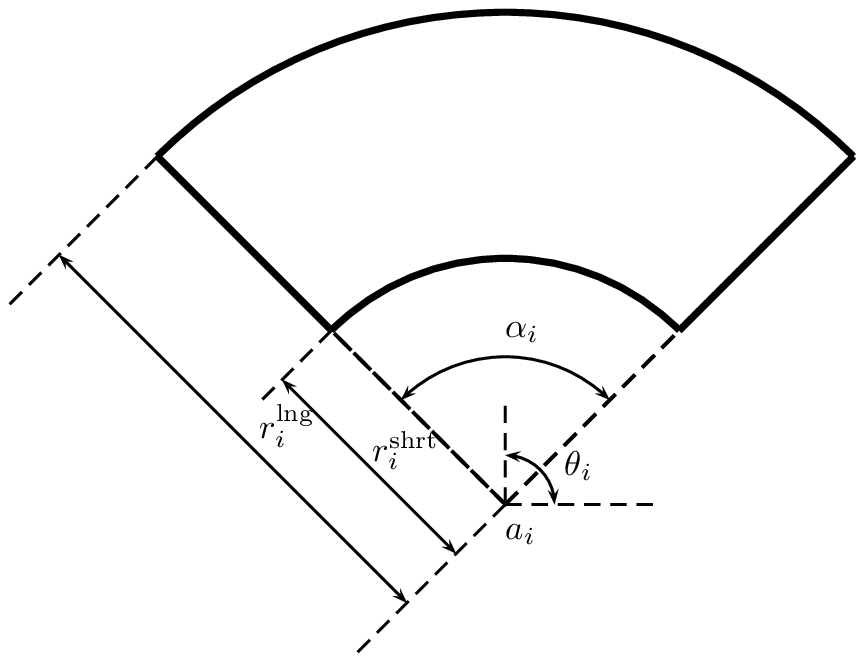},\epsfxsize=3.5in \epsffile{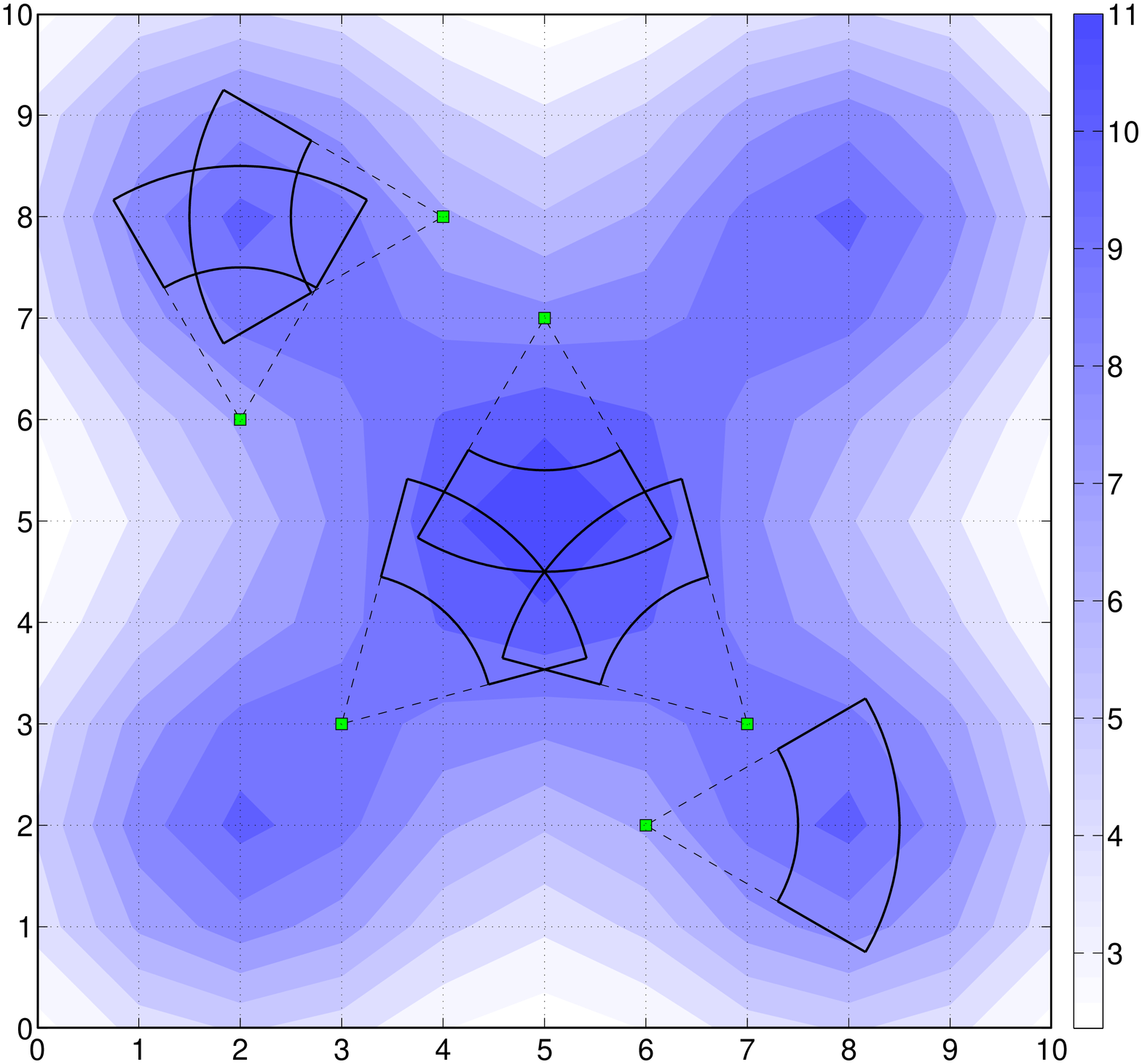}}
  \caption{Visual sensor footprint and a configuration of the mobile sensor network}\label{fig1}
\end{figure}

The visual sensor footprint is completely characterized by the
following parameters: the position of agent~$i$, $a_i\in \QQ$, the
camera orientation, $\theta_i \in [0,2\pi)$, the camera angle of view,
$\alpha_i \in [\subscr{\alpha}{min},\subscr{\alpha}{max}]$, and the
shortest range (resp.~longest range) between agent $i$ and the nearest
(resp.~farthest) object that can be recognized from the image,
$\supscr{r}{shrt}_i \in [\subscr{r}{min},\subscr{r}{max}]$
(resp.~$\supscr{r}{lng}_i \in [\subscr{r}{min},\subscr{r}{max}]$).
The parameters $\supscr{r}{shrt}_i$, $\supscr{r}{lng}_i$, $\alpha_i$
can be tuned by changing the focal length $\FL_i$ of agent $i$'s
camera.  In this way, $c_i : = (\FL_i, \theta_i)\in
[0,\subscr{\FL}{max}]\times [0,2\pi)$ is the camera control vector of
agent $i$. In what follows, we will assume that $c_i$ takes values in
a finite subset $\CC \subset [0,\subscr{\FL}{max}]\times [0,2\pi)$. An
agent action is thus a vector $s_i:=(a_i, c_i)\in {\mathcal{A}}_i :=
\QQ\times \CC$, and a multi-agent action is denoted by
$s=(s_1,\dots,s_N) \in \Ac := \Pi_{i=1}^N
\Ac_i$. 

Let ${\mathcal{D}}(a_i,c_i)$ be the visual sensor footprint of
agent~$i$. Now we can define a proximity sensing
graph\footnote{See~\cite{FB-JC-SM:08} for a definition of proximity
  graph.} $\subscr{\GG}{sen}(s) := (V, \subscr{E}{sen}(s))$ as
follows:
the set of neighbors of agent $i$, $\supscr{\NN}{sen}_i(s)$, is given as
$\supscr{\NN}{sen}_i(s) := \{j \in V\backslash\{i\} \;|\;
{\mathcal{D}}(a_i, c_i)\cap{\mathcal{D}}(a_j,
c_j)\cap{\mathcal{Q}}\neq\emptyset\}$.

Each agent is able to communicate with others to exchange information.
We assume that the communication range of agents is
$2\subscr{r}{max}$.  This induces a $2\subscr{r}{max}$-disk
communication graph $\subscr{\GG}{comm}(s) := (V,
\subscr{E}{comm}(s))$ as follows: the set of neighbors of agent $i$ is
given by $\supscr{\mathcal{N}}{comm}_i(s) := \setdef{j \in
  V\backslash\{i\}} {(x_i - x_j)^2+(y_i - y_j)^2 \leq
  (2\subscr{r}{max})^2}$. Note that $\subscr{\GG}{comm}(s)$ is
undirected and that $\subscr{\GG}{sen}(s)\subseteq
\subscr{\GG}{comm}(s)$.

The motion of agents will be limited to a neighboring point in
$\subscr{\GG}{loc}$ at each time step. Thus, an agent feasible action
set will be given by
${\FF}_i(a_i):=(\{a_i\}\cup\supscr{\NN}{loc}_{a_i})\times
{\mathcal{C}}$.

\subsubsection{Coverage game}

We now proceed to formulate a coverage optimization problem as a
constrained strategic game. For each $q\in{\QQ}$, we denote $n_q(s)$ as
the cardinality of the set $\{k\in V \;|\; q\in {\mathcal{D}}(a_k,
c_k)\cap{\mathcal{Q}}\}$; i.e., the number of agents which can observe
the point $q$.  The ``profit'' given by $W_{q}$ will be equally shared
by agents that can observe the point $q$.  The benefit that agent $i$
obtains through sensing is thus defined by $\sum_{q
  \in{\mathcal{D}}(a_i,c_i)\cap{\mathcal{Q}}}\frac{W_{q}}{n_{q}(s)}$.

On the other hand, and as argued in~\cite{CBM-VP-KO-RM:06}, the
processing of visual data can incur a higher cost than that of
communication. This is in contrast with scalar sensor networks, where
the communication cost dominates.
With this observation, we model the energy consumption of agent $i$ by
$f_i(c_i) := \frac{1}{2}\alpha_i
((\supscr{r}{lng}_i)^2-(\supscr{r}{shrt}_i)^2)$. This measure
corresponds to the area of the visual sensor footprint and can serve
to approximate the energy consumption or the cost incurred by image
processing algorithms.

We will endow each agent with a utility function that aims to capture
the above sensing/processing trade-off.  In this way, we define a
utility function for agent $i$ by
\begin{align*}
  u_i(s) = \sum_{q\in{\mathcal{D}}(a_i,
    c_i)\cap{\mathcal{Q}}}\frac{W_{q}}{n_{q}(s)} - f_i(c_i).
\end{align*}
Note that the utility function $u_i$ is local over the visual sensing
graph $\subscr{\mathcal{G}}{sen}(s)$; i.e., $u_i$ is only dependent on
the actions of $\{i\}\cup \supscr{\NN}{sen}_i(s)$.
With the set of utility functions $\subscr{U}{cov}=\{u_i\}_{i\in V}$,
and feasible action set $\subscr{\FF}{cov} = \Pi_{i=1}^N\bigcup_{a_i
  \in \Ac_i} \FF_i(a_i)$, we now have all the ingredients to introduce
the coverage game $\subscr{\Gamma}{cov} := \langle V, {\mathcal{A}},
\subscr{U}{cov}, \subscr{\FF}{cov}\rangle$. This game is a variation
of the congestion games introduced in~\cite{RWR:73}.

\begin{lemma}
  The coverage game $\subscr{\Gamma}{cov}$ is a constrained potential
  game with potential function
\begin{align*}
  \phi(s) = \sum_{q\in {\mathcal{Q}}}\sum_{\ell=1}^{n_q(s)}\frac{W_q}{\ell}-\sum_{i=1}^N
  f_i(c_i).\end{align*}\label{lem1}
\end{lemma}

\begin{proof} The proof is a slight variation of that
  in~\cite{RWR:73}. Consider any $s:=(s_i, s_{-i})\in {\mathcal{A}}$
  where $s_i := (a_i, c_i)$. We fix $i\in V$ and pick any $s_i'=(a_i',
  c_i')$ from ${\FF}_i(a_i)$. Denote $s':=(s_i', s_{-i})$, $\Omega_1
  := ({\mathcal{D}}(a_i, c_i)\backslash {\mathcal{D}}(a_i',
  c_i'))\cap{\mathcal{Q}}$ and $\Omega_2 := ({\mathcal{D}}(a_i',
  c_i')\backslash {\mathcal{D}}(a_i, c_i))\cap{\mathcal{Q}}$. Observe
  that
\begin{align*}& \phi(s_i, s_{-i})-\phi(s_i',
s_{-i})\nnum\\
&= \sum_{q\in \Omega_1}(\sum_{\ell=1}^{n_q(s)}\frac{W_q}{\ell}-\sum_{\ell=1}^{n_q(s')}\frac{W_q}{\ell})
+\sum_{q\in \Omega_2}(-\sum_{\ell=1}^{n_q(s)}\frac{W_q}{\ell}+\sum_{\ell=1}^{n_q(s')}\frac{W_q}{\ell})
-f_i(c_i)+f_i(c_i')\nnum\\
&= \sum_{q\in \Omega_1}\frac{W_q}{n_q(s)}-\sum_{q\in \Omega_2}\frac{W_q}{n_q(s')}-f_i(c_i)+f_i(c_i')\nnum\\
&= u_i(s_i,s_{-i})-u_i(s_i',s_{-i})
\end{align*} where in the second equality we utilize the fact that for
each $q\in \Omega_1$, $n_q(s) = n_q(s')+1$, and each $q\in \Omega_2$, $n_q(s')= n_q(s)+1$.
\end{proof}

We denote by ${\mathcal{E}}(\subscr{\Gamma}{cov})$ the set of constrained
NEs of $\subscr{\Gamma}{cov}$. It is worth mentioning that
${\mathcal{E}}(\subscr{\Gamma}{cov})\neq\emptyset$ due to the fact that
$\subscr{\Gamma}{cov}$ is a constrained potential game.

\begin{remark} The assumptions of our problem formulation admit
  several extensions. For example, it is straightforward to extend our
  results to non-convex 3-D spaces. This is because the results that
  follow can also handle other shapes of the sensor footprint; e.g.,~a
  complete disk, a subset of the annulus sector. On the other hand,
  note that the coverage problem can be interpreted as a target
  assignment problem---here, the value $W_q\geq0$ would be associated
  with the value of a target located at the point
  $q$.\oprocend\label{rem1}
\end{remark}

\subsection{Notations}
In the following, we will use the Landau symbol, $O$, as in
$O(\epsilon^k)$, for some $k\geq0$. This implies that
$0<\lim_{\epsilon\rightarrow0^+}\frac{O(\epsilon^k)}{\epsilon^k}<+\infty$.  We denote by $\diag{\Ac} :=
\setdef{(s,s)\in {\mathcal{A}}^2}{s\in \mathcal{A}}$ and
$\diag{\EE(\subscr{\Gamma}{cov})}:=\setdef{(s, s)\in
  {\mathcal{A}}^{2}}{ s\in {\mathcal{E}}(\Gamma_{\rm cov})}$.

Consider $a, a'\in {\mathcal{Q}^N}$ where $a_i\neq a_i'$ and
$a_{-i}=a_{-i}'$ for some $i\in V$. The transition $a\rightarrow a'$
is feasible if and only if $(a_i, a_i')\in \subscr{E}{loc}$. A
feasible path from $a$ to $a'$ consisting of multiple feasible
transitions is denoted by $a\Rightarrow a'$. Let $\diamond a :=
\setdef{a'\in\mathcal{Q}}{a \Rightarrow a'}$ be the reachable set from
$a$.

Let $s = (a, c), s' = (a', c') \in {\mathcal{A}}$ where $a_i \neq
a_i'$ and $a_{-i} = a_{-i}'$ for some $i\in V$. The transition $s
\rightarrow s'$ is feasible if and only if $s_i' \in {\FF}_i(a)$. A
feasible path from $s$ to $s'$ consisting of multiple feasible
transitions is denoted by $s \Rightarrow s'$. Finally, $\diamond s :=
\setdef{s' \in {\mathcal{A}}}{ s \Rightarrow s'}$ will be the
reachable set from $s$.


\section{Distributed coverage learning algorithms and convergence
  results}\label{sec:algorithm}

In our coverage problem, we assume that $W_q$ is unknown in
advance. Furthermore, due to the limitations of motion and sensing,
each agent is unable to obtain the information of $W_q$ if the point
$q$ is outside its sensing range. These information constraints
renders that each agent is unable to access to the utility values
induced by alternative actions. Thus the action-based learning
algorithms; e.g., better (or best) reply learning algorithm and
adaptive play learning algorithm can not be employed to solve our
coverage games. It motivates us to design distributed learning
algorithms which only require the payoff received.

In this section, we come up with two distributed payoff-based learning
algorithms, say \emph{Distributed Inhomogeneous Synchronous Coverage
  Learning Algorithm} ($\inhomosyn$, for short) and \emph{Distributed
  Inhomogeneous Asynchronous Coverage Learning Algorithm} ($\inhomoasyn$, for
short). We then present their convergence properties. Relevant
algorithms include payoff-based learning algorithms proposed
in~\cite{JRM-JSS:08}\cite{JRM-HPY-GA-JSS:06}.

\subsection{Distributed Inhomogeneous Synchronous Coverage Learning Algorithm}

For each $t\ge 1$ and $i \in V$, we define $\tau_i(t)$ as follows: $\tau_i(t) = t$ if $u_i(s(t))\geq u_i(s(t-1))$, otherwise, $\tau_i(t) = t-1$. Here, $s_i(\tau_i(t))$ is the more successful action of agent $i$ in last two steps. The main steps of the
$\inhomosyn$ algorithm are the following:

\begin{algorithmic}[1]

  \STATE [\underline{Initialization}:] At $t=0$, all agents are
  uniformly placed in $\mathcal{Q}$. Each agent $i$ uniformly chooses
  its camera control vector $c_i$ from the set $\mathcal{C}$,
  communicates with agents in $\supscr{\NN}{sen}_i(s(0))$, and
  computes $u_i(s(0))$.  At $t = 1$, all the agents keep their
  actions.

  \STATE [\underline{Update}:] At each time $t\geq2$, each agent $i$
  updates its state according to the following rules:
  \begin{itemize}
  \item Agent $i$ chooses the exploration rate $\epsilon(t) =
    t^{-\frac{1}{N(D+1)}}$ and compute $s_i(\tau_i(t))$.
  \item With probability $\epsilon(t)$, agent $i$ experiments, and
    chooses the temporary action $\supscr{s}{tp}_i:=
    (\supscr{a}{tp}_i, \supscr{c}{tp}_i)$ uniformly from the set
    $\FF_i(a_i(t))\setminus\{s_i(\tau_i(t))\}$.
  \item With probability $1-\epsilon(t)$, agent $i$ does not
    experiment, and sets $\supscr{s}{tp}_i=s_i(\tau_i(t))$.
  \item After $\supscr{s}{tp}_i$ is chosen, agent $i$ moves to the
    position $\supscr{a}{tp}_i$ and sets the camera control vector to
    $\supscr{c}{tp}_i$.

  \end{itemize}

  \STATE [\underline{Communication and computation}:] At position
  $\supscr{a}{tp}_i$, agent $i$ communicates with agents in
  $\supscr{\NN}{sen}_i(\supscr{s}{tp}_i, \supscr{s}{tp}_{-i})$, and
  computes $u_i(\supscr{s}{tp}_i, \supscr{s}{tp}_{-i})$ and
  $\FF_i(\supscr{a}{tp}_i)$.

  \STATE Repeat Step 2 and 3.
\end{algorithmic}

\begin{remark}
  A variation of the $\inhomosyn$ algorithm corresponds to
  $\epsilon(t)=\epsilon\in(0,\frac{1}{2}]$ constant for all
  $t\geq2$. If this is the case, we will refer to the algorithm as
  \emph{Distributed Homogeneous Synchronous Coverage Learning
    Algorithm} ($\homosyn$, for short). Later, the convergence analysis of
  the $\inhomosyn$ algorithm will be based on the analysis of the $\homosyn$
  algorithm.\oprocend\label{rem2}
\end{remark}

Denote the space ${\mathcal{B}} :=
\setdef{(s,s')\in{\mathcal{A}}\times{\mathcal{A}}}{s_i'\in{\FF}_i(a_i),\;\forall
  i\in V}$. Observe that $z(t) := (s(t-1), s(t))$ in the $\inhomosyn$
algorithm constitutes a time-inhomogeneous Markov chain
$\{{\mathcal{P}}_t\}$ on the space ${\mathcal{B}}$. The following
theorem states that the $\inhomosyn$ algorithm asymptotically converges to
the set of $\EE(\subscr{\Gamma}{cov})$ in probability.

\begin{theorem}
  Consider the Markov chain $\{{\mathcal{P}}_t\}$ induced by the $\inhomosyn$
  Algorithm. It holds that
  $\lim_{t\rightarrow+\infty}{\mathbb{P}}(z(t)\in\diag{\EE(\subscr{\Gamma}{cov})})
  = 1$.\label{the8}
\end{theorem}

The proofs of Theorem~\ref{the8} are provided in Section~\ref{sec:analysis}.

\begin{remark} The $\inhomosyn$ algorithm is simpler than the payoff-based
  learning algorithm proposed in~\cite{JRM-HPY-GA-JSS:06}, reducing
  the computational complexity. The algorithm studied
  in~\cite{JRM-HPY-GA-JSS:06} converges the set of NEs with a
  arbitrarily high probability by choosing a arbitrarily small
  exploration rate $\epsilon$ in advance.  However, it is difficult to
  derive the relation between the convergent probability and the
  exploration rate. It motivates us to utilize a diminishing
  exploration rate in the $\inhomosyn$ algorithm which induces a
  time-inhomogeneous Markov chain in contrast to a time-inhomogeneous
  Markov chain in~\cite{JRM-HPY-GA-JSS:06}. This change renders a
  stronger convergence property, i.e., the convergence to the set of
  NEs in probability. \oprocend\label{rem3}
\end{remark}

\subsection{Distributed Inhomogeneous Asynchronous Coverage Learning
  Algorithm}

Lemma~\ref{lem1} shows that the coverage game $\subscr{\Gamma}{cov}$
is a constrained potential game with potential function
$\phi(s)$. However, this potential function is not a straightforward
measure of the network coverage performance. On the other hand, the
objective function $U_g(s) := \sum_{i\in V}u_i(s)$ captures the
trade-off between the overall network benefit from sensing and the
total energy the network consumes, and thus can be perceived as a more
natural coverage performance metric. Denote by $S^* := \setdef{s}{{\rm
    argmax}_{s\in{\mathcal{A}}}U_g(s)}$ as the set of global
maximizers of $U_g(s)$. In this part, we present the $\inhomoasyn$ algorithm
which is convergent in probability to the set $S^*$.

Before that , we first introduce some notations for the $\inhomoasyn$
algorithm. Denote by ${\mathcal{B}}'$ the space ${{\mathcal{B}}}' :=
\setdef{(s, s')\in{\mathcal{A}}\times{\mathcal{A}}}{s_{-i} =
  s_{-i}',\; s_i'\in {\mathcal{F}}_i(a_i)\; {\rm for\;\; some}\;\;i\in
  V}$. For any $s^0, s^1\in{\mathcal{A}}$ with $s^0_{-i}=s^1_{-i}$ for
some $i\in V$, we denote
\begin{align*}
  \Delta_i(s^1,s^0) := \frac{1}{2}\sum_{q\in\Omega_1}
  \frac{W_q}{n_q(s^1)}-\frac{1}{2}\sum_{q\in\Omega_2}\frac{W_q}{n_q(s^0)},
\end{align*}
where $\Omega_1 :=
{\mathcal{D}}(a_i^1,c_i^1)\backslash{\mathcal{D}}(a_i^0,c_i^0)\cap\QQ$
and $\Omega_2 :=
{\mathcal{D}}(a_i^0,c_i^0)\backslash{\mathcal{D}}(a_i^1,c_i^1)\cap\QQ$,
and
\begin{align*}
  &\rho_i(s^0,s^1) :=
  u_i(s^1)-\Delta_i(s^1,s^0)-u_i(s^0)+\Delta_i(s^0,s^1),\nnum\\&\Psi_i(s^0,
  s^1) := \max\{u_i(s^0)-\Delta_i(s^0,s^1),
  u_i(s^1)-\Delta_i(s^1,s^0)\},\nnum\\ &m^* :=
  \max_{(s^0,s^1)\in{\mathcal{B}},s_i^0\neq s_i^1}\{\Psi_i(s^0,
  s^1)-(u_i(s^0)-\Delta_i(s^0,s^1)),\frac{1}{2}\}.
\end{align*}
It is easy to check that $\Delta_i(s^1,s^0) = -\Delta_i(s^0,s^1)$ and
$\Psi_i(s^0, s^1) = \Psi_i(s^1, s^0)$. Assume that at each time
instant, one of agents becomes active with equal probability. Denote
by $\gamma_i(t)$ the last time instant before $t$ when agent $i$ was
active. We then denote $\gamma^{(2)}_i(t) :=
\gamma_i(\gamma_i(t))$. The main steps of the $\inhomoasyn$ algorithm are
described in the following.

\begin{algorithmic}[1]

  \STATE [\underline{Initialization}:] At $t=0$, all agents are
  uniformly placed in $\QQ$. Each agent $i$ uniformly chooses the
  camera control vector $c_i$ from the set $\CC$, and then
  communicates with agents in $\supscr{\NN}{sen}_i(s(0))$ and computes
  $u_i(s(0))$. Furthermore, each agent $i$ chooses $m_i \in (2m^*,
  Km^*]$ for some $K\geq2$.  At $t = 1$, all the sensors keep their
  actions.

  \STATE [\underline{Update}:] Assume that agent $i$ is active at time
  $t\geq2$. Then agent $i$ updates its state according to the
  following rules:

  $\bullet$ Agent $i$ chooses the exploration rate $\epsilon(t) =
  t^{-\frac{1}{(D+1)(K+1)m^*}}$.

  $\bullet$ With probability $\epsilon(t)^{m_i}$, agent $i$
  experiments and uniformly chooses $\supscr{s}{tp}_i :=
  (\supscr{a}{tp}_i, \supscr{c}{tp}_i)$ from the action set
  $\FF_i(a_i(t))\setminus\{s_i(t), s_i(\gamma_i^{(2)}(t)+1)\}$.

  $\bullet$ With probability $1-\epsilon(t)^{m_i}$, agent $i$ does not
  experiment and chooses $\supscr{s}{tp}_i$ according to the following
  probability distribution:
  \begin{align*}
    &\mathbb{P}(\supscr{s}{tp}_i = s_i(t)) =
    \frac{1}{1+\epsilon(t)^{\rho_i(s_i(\gamma_i^{(2)}(t)+1), s_i(t))}},\nnum\\
    &\mathbb{P}(\supscr{s}{tp}_i = s_i(\gamma_i^{(2)}(t)+1)) =
    \frac{\epsilon(t)^{\rho_i(s_i(\gamma_i^{(2)}(t)+1),
        s_i(t))}}{1+\epsilon(t)^{\rho_i(s_i(\gamma_i^{(2)}(t)+1),
        s_i(t))}}.
  \end{align*}

  $\bullet$ After $\supscr{s}{tp}_i$ is chosen, agent $i$ moves to the
  position $\supscr{a}{tp}_i$ and sets its camera control vector to be
  $\supscr{c}{tp}_i$.

  \STATE [\underline{Communication and computation}:] At position
  $\supscr{a}{tp}_i$, the active agent $i$ communicates with agents in
  $\supscr{\NN}{sen}_i(\supscr{s}{tp}_i,s_{-i}(t))$, and computes
  $u_i(\supscr{s}{tp}_i,s_{-i}(t))$,
  $\Delta_i((\supscr{s}{tp}_i,s_{-i}(t)),s(\gamma_i(t)+1))$,
  $\FF_i(\supscr{a}{tp}_i)$.

  \STATE Repeat Step 2 and 3.
\end{algorithmic}

\begin{remark}
  A variation of the $\inhomoasyn$ algorithm corresponds to
  $\epsilon(t)=\epsilon\in(0,\frac{1}{2}]$ constant for all
  $t\geq2$. If this is the case, we will refer to the algorithm as the
  \emph{Distributed Homogeneous Asynchronous Coverage Learning
    Algorithm} ($\homoasyn$, for short). Later, we will base the
  convergence analysis of the $\inhomoasyn$ algorithm on that of the $\homoasyn$
  algorithm.\oprocend
  \label{rem4}\end{remark}

Like the $\inhomosyn$ algorithm, $z(t) := (s(t-1), s(t))$ in the $\inhomoasyn$
algorithm constitutes a time-inhomogeneous Markov chain
$\{{\mathcal{P}}_t\}$ on the space ${\mathcal{B}}'$. The following
theorem states that the convergence property of the $\inhomoasyn$ algorithm.

\begin{theorem} Consider the Markov chain $\{{\mathcal{P}}_t\}$
  induced by the $\inhomoasyn$ algorithm for the game
  $\subscr{\Gamma}{cov}$. Then it holds that
  $\lim_{t\rightarrow+\infty}{\mathbb{P}}(z(t)\in\diag{S^*}) =
  1$.\label{the2}
\end{theorem}

The proofs of Theorem~\ref{the2} are provided in Section~\ref{sec:analysis}.

\begin{remark} The authors in~\cite{JRM-JSS:08} proposed a synchronous
  payoff-based log-linear learning algorithm. This algorithm is able
  to maximize the potential function of a potential game. While the
  $\inhomoasyn$ algorithm is a variation of that in~\cite{JRM-JSS:08}, and optimizes a different function $U_g(s)$. Furthermore, the convergence of the $\inhomoasyn$ algorithm is in probability and stronger than the arbitrarily high
  probability~\cite{JRM-JSS:08} by choosing an arbitrarily small
  exploration rate in advance.\oprocend\label{rem5}\end{remark}

\section{Convergence Analysis}\label{sec:analysis}

In this section, we prove Theorem~\ref{the8} and~\ref{the2} by
appealing to the Theory of Resistance Trees in~\cite{HPY:93} and the
results in strong ergodicity in~\cite{DI-RM:76}. Relevant papers
include~\cite{JRM-JSS:08}\cite{JRM-HPY-GA-JSS:06} where the Theory of
Resistance Trees in~\cite{HPY:93} is first utilized to study the class
of payoff-based learning algorithms, and~\cite{BG:85}\cite{SA-AF:87}\cite{DM-FR-ASV:86} where the strong ergodicity theory is employed to characterize the convergence properties of time-inhomogeneous Markov chains.

\subsection{Convergence analysis of the $\inhomosyn$ Algorithm}

We first utilize Theorem~\ref{the4} to characterize the convergence
properties of the associated $\homosyn$ algorithm. This is essential for
the analysis of the $\inhomosyn$ algorithm.

Observe that $z(t):=(s(t-1), s(t))$ in the $\homosyn$ algorithm constitutes
a time-homogeneous Markov chain $\{{\mathcal{P}}^{\epsilon}_t\}$ on
the space ${\mathcal{B}}$. Consider $z, z'\in {\mathcal{B}}$. A
feasible path from $z$ to $z'$ consisting of multiple feasible
transitions of $\{{\mathcal{P}}^{\epsilon}_t\}$ is denoted by $z
\Rightarrow z'$. The reachable set from $z$ is denoted as $\diamond z
:= \{z' \in {\mathcal{B}} \;|\; z \Rightarrow z'\}$.

\begin{lemma}
  $\{{\mathcal{P}}^{\epsilon}_t\}$ is a regular perturbation of
  $\{{\mathcal{P}}^0_t\}$.\label{lem2}
\end{lemma}

\begin{proof}
  Consider a feasible transition $z^1 \rightarrow z^2$ with
  $z^1:=(s^0,s^1)$ and $z^2:=(s^1,s^2)$. Then we can define a
  partition of $V$ as $\Lambda_1:=\setdef{i\in V}{s^2_i =
    s_i^{\tau_i(0, 1)}}$ and $\Lambda_2:=\setdef{i\in V}{s^2_i \in
    {\FF}_i(a^1_i)\setminus\{s^{\tau_i(0, 1)}_i\}}$. The corresponding
  probability is given by
\begin{align}
  P^{\epsilon}_{z^1 z^2} = \prod_{i\in
    \Lambda_1}(1-\epsilon)\times\prod_{j\in
    \Lambda_2}\frac{\epsilon}{|\FF_i(a^1_i)|-1}.\label{e7}
\end{align}

Hence, the resistance of the transition $z^1 \rightarrow z^2$ is
$|\Lambda_2|\in\{0,1,\cdots,N\}$ since \begin{align*}
  0<\lim_{\epsilon\rightarrow 0^+} \frac{P^{\epsilon}_{z^1
      z^2}}{\epsilon^{|\Lambda_2|}}=\prod_{j\in
    \Lambda_2}\frac{1}{|\FF_i(a^1_i)|-1}<+\infty.\end{align*}

We have that (A3) in Section~\ref{sec:resistancetrees} holds. It is
not difficult to see that (A2) holds, and we are now in a position to
verify (A1). Since $\subscr{\GG}{loc}$ is undirected and connected,
and multiple sensors can stay in the same position, then $\diamond a^0
= {\mathcal{Q}}^N$ for any $a^0\in {\mathcal{Q}}$. Since sensor $i$
can choose any camera control vector from $\mathcal{C}$ at each time, then $\diamond s^0 =
{\mathcal{A}}$ for any $s^0\in{\mathcal{A}}$. It implies that
$\diamond z^0 = {\mathcal{B}}$ for any $z^0 \in {\mathcal{B}}$, and
thus the Markov chain $\{{\mathcal{P}}^{\epsilon}_t\}$ is irreducible
on the space $\mathcal{B}$.

It is easy to see that any state in $\diag{\Ac}$ has period $1$. Pick
any $(s^0, s^1)\in{\mathcal{B}}\setminus\diag{\Ac}$. Since
$\subscr{\mathcal{G}}{loc}$ is undirected, then $s_i^0\in
{\FF}_i(a_i^1)$ if and only if $s_i^1\in {\FF}_i(a_i^0)$. Hence, the
following two paths are both feasible:
\begin{align*}
  &(s^0, s^1)\rightarrow (s^1, s^0) \rightarrow (s^0, s^1)\nnum\\
  &(s^0, s^1)\rightarrow (s^1, s^1) \rightarrow (s^1, s^0) \rightarrow
  (s^0, s^1).
\end{align*}
Hence, the period of the state $(s^0,s^1)$ is $1$. This proves
aperiodicity of $\{{\mathcal{P}}^{\epsilon}_t\}$. Since
$\{{\mathcal{P}}^{\epsilon}_t\}$ is irreducible and aperiodic, then
(A1) holds.
\end{proof}

\begin{lemma} For any $(s^0, s^0)\in
  \diag{\Ac}\setminus\diag{\EE(\subscr{\Gamma}{cov})}$, there is a
  finite sequence of transitions from $(s^0,s^0)$ to some $(s^*,
  s^*)\in\diag{\EE(\subscr{\Gamma}{cov})}$ that
  satisfies
  \begin{align*}
    &{\mathcal{L}}:=(s^0, s^0)\stackrel{O(\epsilon)}{\rightarrow}
    (s^0, s^1)\stackrel{O(1)}{\rightarrow} (s^1,
    s^1)\stackrel{O(\epsilon)}{\rightarrow}(s^1,
    s^2)\nnum\\&\stackrel{O(1)}{\rightarrow} (s^2,
    s^2)\stackrel{O(\epsilon)}{\rightarrow}
    \cdots\stackrel{O(\epsilon)}{\rightarrow}(s^{k-1},s^k)\stackrel{O(1)}{\rightarrow}(s^k,s^k)
  \end{align*}
  where $(s^k,s^k)=(s^*, s^*)$ for some
  $k\geq1$.\label{lem5}
\end{lemma}

\begin{proof}
  If $s^0\notin{\mathcal{E}}(\subscr{\Gamma}{cov})$, there exists a
  sensor $i$ with a action $s_i^1\in{\FF}_i(a_i^0)$ such that
  $u_i(s^1) > u_i(s^0)$ where $s_{-i}^0 = s_{-i}^1$. The transition
  $(s^0, s^0) \rightarrow (s^0, s^1)$ happens when only sensor $i$
  experiments, and its corresponding probability is
  $(1-\epsilon)^{N-1}\times\frac{\epsilon}{|{\FF}_i(a_i^0)|-1}$. Since
  the function $\phi$ is the potential function of the game
  $\subscr{\Gamma}{cov}$, then we have that $\phi(s^1) - \phi(s^0) =
  u_i(s^1) - u_i(s^0)$ and thus $\phi(s^1) > \phi(s^0)$.

  Since $u_i(s^1) > u_i(s^0)$ and $s_{-i}^0 = s_{-i}^1$, the
  transition $(s^0, s^1) \rightarrow (s^1, s^1)$ occurs when all
  sensors do not experiment, and the associated probability is
  $(1-\epsilon)^N$.

  We repeat the above process and construct the path $\mathcal{L}$
  with length $k\geq1$. Since $\phi(s^i) > \phi(s^{i-1})$ for $i
  =\until{k}$, then
  $s^i\neq s^j$ for $i\neq j$ and thus the path $\mathcal{L}$ has no
  loop. Since $\mathcal{A}$ is finite, then $k$ is finite and thus
  $s^k=s^*\in{\mathcal{E}}(\subscr{\Gamma}{cov})$.
\end{proof}


A direct result of Lemma~\ref{lem2} is that for each $\epsilon$, there
exists a unique stationary distribution of
$\{{\mathcal{P}}^{\epsilon}_t\}$, say $\mu(\epsilon)$. We now proceed
to utilize Theorem~\ref{the4} to characterize
$\lim_{\epsilon\rightarrow0^+}\mu(\epsilon)$.

\begin{proposition} Consider the regular perturbation
  $\{{\mathcal{P}}^{\epsilon}_t\}$ of $\{{\mathcal{P}}^0_t\}$. Then
  $\displaystyle{\lim_{\epsilon\rightarrow0^+}\mu(\epsilon)}$ exists and the limiting
  distribution $\mu(0)$ is a stationary distribution of
  $\{{\mathcal{P}}^0_t\}$. Furthermore, the stochastically stable
  states (i.e., the support of $\mu(0)$) are contained in the set
  $\diag{\EE(\subscr{\Gamma}{cov})}$.\label{pro4}
\end{proposition}

\begin{proof}
  Notice that the stochastically stable states are contained in the
  recurrent communication classes of the unperturbed Markov chain that
  corresponds to the $\homosyn$ Algorithm with $\epsilon = 0$. Thus the
  stochastically stable states are included in the set
  $\diag{\Ac}\subset {\mathcal{B}}$. Denote by $T_{\min}$ the minimum
  resistance tree and by $h_v$ the root of $T_{\min}$. Each edge of
  $T_{\min}$ has resistance $0,1,2,\dots$ corresponding to the
  transition probability $O(1),O(\epsilon),O(\epsilon^2),\dots$. The
  state $z'$ is the $successor$ of the state $z$ if and only if $(z,
  z')\in T_{\min}$. Like Theorem 3.2 in~\cite{JRM-HPY-GA-JSS:06}, our
  analysis will be slightly different from the presentation
  in~\ref{sec:resistancetrees}. We will construct $T_{\min}$ over
  states in the set $\mathcal{B}$ (rather than $\diag{\Ac}$) with the
  restriction that all the edges leaving the states in
  ${\mathcal{B}}\setminus\diag{\Ac}$ have resistance $0$. The
  stochastically stable states are not changed under this
  difference.

  \begin{claim}
    For any $(s^0,s^1)\in {\mathcal{B}}\setminus\diag{\Ac}$, there is
    a finite path
    \begin{align*}
      {\mathcal{L}}' := (s^0, s^1)\stackrel{O(1)}{\rightarrow} (s^1,
      s^2)\stackrel{O(1)}{\rightarrow} (s^2, s^2)
  \end{align*}
  where $s_i^2 = s_i^{\tau_i(0, 1)}$ for all $i\in V$.
  \end{claim}
  \begin{proof}
    These two transitions occur when all agents do not experiment. The
    corresponding probability of each transition is $(1-\epsilon)^N$.
  \end{proof}

  \begin{claim} The root $h_v$ belongs to the set $\diag{\Ac}$.\end{claim}

  \begin{proof} Suppose that $h_v=(s^0, s^1)\in
  {\mathcal{B}}\setminus\diag{\Ac}$. By Claim 1, there is a finite
  path ${\mathcal{L}}' := (s^0, s^1)\stackrel{O(1)}{\rightarrow} (s^1,
  s^2)\stackrel{O(1)}{\rightarrow} (s^2, s^2)$. We now construct a new
  tree $T'$ by adding the edges of the path ${\mathcal{L}}'$ into the
  tree $T_{\min}$ and removing the redundant edges. The total
  resistance of adding edges is $0$. Observe that the resistance of
  the removed edge exiting from $(s^2,s^2)$ in the tree $T_{\min}$ is
  at least $1$. Hence, the resistance of $T'$ is strictly lower than
  that of $T_{\min}$, and we get to a contradiction.\end{proof}

  \begin{claim} Pick any ${s}^* \in
  {\mathcal{E}}(\subscr{\Gamma}{cov})$ and consider $z := ({s}^*,
  {s}^*),$ $z' := ({s}^*,{\tilde{s}})$ where $\tilde{s}\neq s^*$. If
  $(z, z')\in T_{\min}$, then the resistance of the edge $(z, z')$ is
  some $k\geq2$.\end{claim}

  \begin{proof} Suppose the deviator in the transition $z\rightarrow
  z'$ is unique, say $i$. Then the corresponding transition
  probability is $O(\epsilon)$. Since
  $s^*\in{\mathcal{E}}(\subscr{\Gamma}{cov})$ and
  $\tilde{s}_i\in{\FF}_i(a_i^*)$, we have that $u_i(s_i^*, s_{-i}^*)
  \geq u_i(\tilde{s}_i,\tilde{s}_{-i})$, where $s_{-i}^* =
  \tilde{s}_{-i}$.

  Since $z'\in {\mathcal{B}}\setminus\diag{\Ac}$, it follows from
  Claim 2 that the state $z'$ can not be the root of $T_{\min}$ and
  thus has a successor $z''$. Note that all the edges leaving the
  states in ${\mathcal{B}}\setminus\diag{\Ac}$ have resistance
  $0$. Then none experiments in the transition $z'\rightarrow z''$ and
  $z'' = (\tilde{s}, \hat{s})$ for some $\hat{s}$. Since $u_i(s_i^*,
  s_{-i}^*) \geq u_i(\tilde{s}_i,\tilde{s}_{-i})$ with $s_{-i}^* =
  \tilde{s}_{-i}$, we have $\hat{s} = s^*$ and thus $z'' = (\tilde{s},
  s^*)$. Similarly, the state $z''$ must have a successor $z'''$ and
  $z'''=z$. We then obtain a loop in $T_{\min}$ which contradicts that
  $T_{\min}$ is a tree.

  It implies that at least two sensors experiment in the transition
  $z\rightarrow z'$. Thus the resistance of the edge $(z, z')$ is at
  least 2.\end{proof}

  \begin{claim} The root $h_v$ belongs to the set
  $\diag{\EE(\subscr{\Gamma}{cov})}$.\end{claim}

  \begin{proof} Suppose that $h_v=(s^0, s^0)\notin
  \diag{\EE(\subscr{\Gamma}{cov})}$. By Lemma~\ref{lem5}, there is a
  finite path $\mathcal{L}$ connecting $(s^0, s^0)$ and some $(s^*,
  s^*)\in\diag{\EE(\subscr{\Gamma}{cov})}$. We now construct a new
  tree $T'$ by adding the edges of the path $\mathcal{L}$ into the
  tree $T_{\min}$ and removing the edges that leave the states in
  $\mathcal{L}$ in the tree $T_{\min}$. The total resistance of adding
  edges is $k$. Observe that the resistance of the removed edge
  exiting from $(s^i,s^i)$ in the tree $T_{\min}$ is at least $1$ for
  $i\in\{1,\cdots,k-1\}$. By Claim 3, the resistance of the removed
  edge leaving from $(s^*,s^*)$ in the tree $T_{\min}$ is at least
  $2$. The total resistance of removing edges is at least
  $k+1$. Hence, the resistance of $T'$ is strictly lower than that of
  $T_{\min}$, and we get to a contradiction.\end{proof}

  It follows from Claim 4 that the states in
  $\diag{\EE(\subscr{\Gamma}{cov})}$ have minimum stochastic
  potential. Since Lemma~\ref{lem2} shows that Markov chain
  $\{{\mathcal{P}}^{\epsilon}_t\}$ is a regularly perturbed Markov
  process, Proposition~\ref{pro4} is a direct result of
  Theorem~\ref{the4}.
\end{proof}

We are now ready to show Theorem~\ref{the8}.

\textbf{Proof of Theorem~\ref{the8}:}

\begin{claim} Condition (B2) in Theorem~\ref{the5} holds.\end{claim}

\begin{proof}
For each $t\geq0$ and each $z\in X$, we defines the
numbers \begin{align*} &\sigma_z(\epsilon(t)) :=
  \sum_{T\in{G(z)}}\prod_{(x,y)\in T}P^{\epsilon (t)}_{xy},\quad
  \sigma_z^t = \sigma_z(\epsilon(t))\nnum\\ &\mu_z(\epsilon(t)) :=
  \frac{\sigma_z(\epsilon(t))}{\sum_{x\in
      X}\sigma_x(\epsilon(t))},\quad \mu_z^t =
  \mu_z(\epsilon(t)). \end{align*}

Since $\{{\mathcal{P}}^{\epsilon}_t\}$ is a regular perturbation of
$\{{\mathcal{P}}^0_t\}$, then it is irreducible and thus $\sigma_z^t >
0$. As Lemma 3.1 of Chapter 6 in~\cite{MF-AW:50}, one can show that
$(\mu^t)^TP^{\epsilon (t)} = (\mu^t)^T$. Therefore, condition (B2) in
Theorem~\ref{the5} holds.\end{proof}

\begin{claim} Condition (B3) in Theorem~\ref{the5} holds.\end{claim}

\begin{proof}
We now proceed to verify condition (B3) in Theorem~\ref{the5}. To do
that, let us first fix $t$, denote $\epsilon = \epsilon(t)$ and study
the monotonicity of $\mu_z(\epsilon)$ with respect to $\epsilon$. We
write $\sigma_z(\epsilon)$ in the following
form
\begin{align}
  \sigma_z(\epsilon) = \sum_{T\in G(z)}\prod_{(x,y)\in
    T}P^{\epsilon}_{xy} = \sum_{T\in G(z)}\prod_{(x,y)\in
    T}\frac{\alpha_{xy}(\epsilon)}{\beta_{xy}(\epsilon)} =
  \frac{\alpha_z(\epsilon)}{\beta_z(\epsilon)}\label{e6}
\end{align}
for some polynomials $\alpha_z(\epsilon)$ and $\beta_z(\epsilon)$ in
$\epsilon$. With~\eqref{e6} in hand, we have that $\sum_{x\in
  X}\sigma_x(\epsilon)$ and thus $\mu_z(\epsilon)$ are ratios of two
polynomials in $\epsilon$; i.e., $\mu_z(\epsilon) =
\frac{\varphi_z(\epsilon)}{\beta(\epsilon)}$ where
$\varphi_z(\epsilon)$ and $\beta(\epsilon)$ are polynomials in
$\epsilon$. The derivative of $\mu_z(\epsilon)$ is given by
\begin{align*}
\frac{\partial \mu_z(\epsilon)}{\partial \epsilon} =
  \frac{1}{\beta(\epsilon)^2}(\frac{\partial
    \varphi_z(\epsilon)}{\partial \epsilon}\beta(\epsilon) -
  \varphi_z(\epsilon)\frac{\partial \beta(\epsilon)}{\partial
    \epsilon}).
\end{align*}

Note that the numerator $\frac{\partial \varphi_z(\epsilon)}{\partial
  \epsilon}\beta(\epsilon) - \varphi_z(\epsilon)\frac{\partial
  \beta(\epsilon)}{\partial \epsilon}$ is a polynomial in
$\epsilon$. Denote by $\iota_z\neq0$ the coefficient of the leading
term of $\frac{\partial \varphi_z(\epsilon)}{\partial \epsilon} -
\varphi_z(\epsilon)\frac{\partial \beta(\epsilon)}{\epsilon}$. The
leading term dominates $\frac{\partial \varphi_z(\epsilon)}{\partial
  \epsilon} - \varphi_z(\epsilon)\frac{\partial
  \beta(\epsilon)}{\epsilon}$ when $\epsilon$ is sufficiently
small. Thus there exists $\epsilon_z > 0$ such that the sign of
$\frac{\partial \mu_z(\epsilon)}{\partial \epsilon}$ is the sign of
$\iota_z$ for all $0< \epsilon \leq \epsilon_z$. Let $\epsilon^* =
\max_{z\in X}\epsilon_z$.

Since $\epsilon(t)$ strictly decreases to zero, then there is a unique
finite time instant $t^*$ such that $\epsilon(t^*) = \epsilon^*$ (if
$\epsilon(0) < \epsilon^*$, then $t^* = 0$). Since $\epsilon(t)$ is
strictly decreasing, we can define a partition of $X$ as
follows:
\begin{align*}
  &\Xi_1 := \{z \in X\; |\; \mu_z(\epsilon(t)) >
  \mu_z(\epsilon(t+1)),\quad \forall t\in [t^*, +\infty)\},\nnum\\
  &\Xi_2 := \{z \in X\; |\; \mu_z(\epsilon(t)) <
  \mu_z(\epsilon(t+1)),\quad \forall t\in[t^*, +\infty)\}.
\end{align*}

We are now ready to verify (B3) of Theorem~\ref{the5}. Since
$\{{\mathcal{P}}^{\epsilon}_t\}$ is a regular perturbed Markov chain
of $\{{\mathcal{P}}^0_t\}$, it follows from Theorem~\ref{the4} that
$\lim_{t\rightarrow+\infty}\mu_z(\epsilon(t))=\mu_z(0)$, and thus it
holds that
\begin{align*}
  & \sum_{t=0}^{+\infty}\sum_{z\in X}\|\mu^t_z-\mu^{t+1}_z\| = \sum_{t=0}^{+\infty}\sum_{z\in X}|\mu_z(\epsilon(t))-\mu_z(\epsilon(t+1))|\nnum\\
  & = \sum_{t=0}^{t^*}\sum_{z\in X}|\mu_z(\epsilon(t))-\mu_z(\epsilon(t+1))| + \sum_{t=t^*+1}^{+\infty}(\sum_{z\in \Xi_1}\mu_z(\epsilon(t))-\sum_{z\in \Xi_1}\mu_z(\epsilon(t+1)))\nnum\\
  &+ \sum_{t=t^*+1}^{+\infty}(1-\sum_{z\in \Xi_1}\mu_z(\epsilon(t+1))-(1-\sum_{z\in \Xi_1}\mu_z(\epsilon(t))))\nnum\\
  &= \sum_{t=0}^{t^*}\sum_{z\in
    X}|\mu_z(\epsilon(t))-\mu_z(\epsilon(t+1))| + 2\sum_{z\in
    \Xi_1}\mu_z(\epsilon(t^*+1)) - 2\sum_{z\in \Xi_1}\mu_z(0)<+\infty.
\end{align*}\end{proof}

\begin{claim} Condition (B1) in Theorem~\ref{the5} holds.\end{claim}

\begin{proof}
Denote by $P^{\epsilon(t)}$ the transition matrix of
$\{{\mathcal{P}}_t\}$. As in~\eqref{e7}, the probability of the
feasible transition $z^1\rightarrow z^2$ is given by
\begin{align*}
  P^{\epsilon(t)}_{z^1 z^2} = \prod_{i\in
    \Lambda_1}(1-\epsilon(t))\times\prod_{j\in
    \Lambda_2}\frac{\epsilon(t)}{|\FF_i(a^1_i)|-1}.
\end{align*}
Observe that $|\FF_i(a^1_i)|\leq 5|{\mathcal{C}}|$. Since
$\epsilon(t)$ is strictly decreasing, there is $t_0\geq 1$ such that
$t_0$ is the first time when $1-\epsilon(t)\geq
\frac{\epsilon(t)}{5|{\mathcal{C}}|-1}$. Then for all $t\geq t_0$, it
holds that
\begin{align*} P^{\epsilon(t)}_{z^1 z^2} \geq
  (\frac{\epsilon(t)}{5|{\mathcal{C}}|-1})^N.
\end{align*}


Denote $P(m,n) := \prod_{t=m}^{n-1}P^{\epsilon(t)}$, $0\leq m<
n$. Pick any $z\in {\mathcal{B}}$ and let $u_z\in{\mathcal{B}}$ be
such that $P_{u_z z}(t, t+D+1) = \min_{x\in {\mathcal{B}}} P_{xz}(t,
t+D+1)$. Consequently, it follows that for all $t\geq t_0$,
\begin{align*}
  & \min_{x\in {\mathcal{B}}} P_{xz}(t, t+D+1)
  = \sum_{i_1\in {\mathcal{B}}}\cdots\sum_{i_D\in\in {\mathcal{B}}}P_{u_zi_1}^{\epsilon(t)}\cdots P_{i_{D-1}i_D}^{\epsilon(t+D-1)}P_{i_D z}^{\epsilon(t+D)}\nnum\\
  &\geq P_{u_zi_1}^{\epsilon(t)}\cdots
  P_{i_{D-1}i_D}^{\epsilon(t+D-1)}P_{i_D z}^{\epsilon(t+D)}\geq
  \prod_{i=0}^D(\frac{\epsilon(t+i)}{5|{\mathcal{C}}|-1})^{N}\geq
  (\frac{\epsilon(t)}{5|{\mathcal{C}}|-1})^{(D+1)N}
\end{align*}
where in the last inequality we use $\epsilon(t)$ begin
strictly decreasing. Then we have
\begin{align*}
  &1-\lambda(P(t, t+D+1)) = \min_{x,y\in{\mathcal{B}}}\sum_{z\in{\mathcal{B}}}\min\{P_{xz}(t,t+D+1), P_{yz}(t,t+D+1)\}\nnum\\
  &\geq \sum_{z\in{\mathcal{B}}} P_{u_zz}(t, t+D+1)\geq |{\mathcal{B}}|(\frac{\epsilon(t)}{5|{\mathcal{C}}|-1})^{(D+1)N}.
\end{align*}
Choose $k_i := (D+1)i$ and let $i_0$ be the smallest integer such that
$(D+1)i_0\geq t_0$. Then, we have that:

\begin{align}
  &\sum_{i=0}^{+\infty}(1-\lambda(P(k_i, k_{i+1})))
  \geq |{\mathcal{B}}|\sum_{i=i_0}^{+\infty}(\frac{\epsilon((D+1)i)}{5|{\mathcal{C}}|-1})^{(D+1)N}\nnum\\
  & =
  \frac{|{\mathcal{B}}|}{(5|{\mathcal{C}}|-1)^{(D+1)N}}\sum_{i=i_0}^{+\infty}\frac{1}{(D+1)i}
  = +\infty.\label{e8}
\end{align}
Hence, the weak ergodicity property follows from Theorem~\ref{the7}.\end{proof}

All the conditions in Theorem~\ref{the5} hold. Thus it follows from Theorem~\ref{the5} that
the limiting distribution is $\mu^* = \lim_{t\rightarrow+\infty}\mu^t$. Note that
$\lim_{t\rightarrow+\infty}\mu^t = \lim_{t\rightarrow+\infty}\mu(\epsilon(t)) = \mu(0)$ and
Proposition~\ref{pro4} shows that the support of $\mu(0)$ is contained
in the set $\diag{\EE(\subscr{\Gamma}{cov})}$. Hence, the support of
$\mu^*$ is contained in the set $\diag{\EE(\subscr{\Gamma}{cov})}$,
implying that $\lim_{t\rightarrow+\infty}{\mathbb{P}}(z(t)\in\diag{\EE(\subscr{\Gamma}{cov})})
= 1$. This completes the proof.

\subsection{Convergence analysis of the $\inhomoasyn$ Algorithm}

First of all, we employ Theorem~\ref{the4} to study the convergence
properties of the associated $\homoasyn$ algorithm. This is essential to
analyze the $\inhomoasyn$ algorithm.

To simplify notations, we will use $s_i(t-1) :=
s_i(\gamma^{(2)}_i(t)+1)$ in the remainder of this section. Observe
that $z(t):=(s(t-1), s(t))$ in the $\homoasyn$ algorithm constitutes a
Markov chain $\{{\mathcal{P}}^{\epsilon}_t\}$ on the space
${\mathcal{B}}'$.

\begin{lemma}
  The Markov chain $\{{\mathcal{P}}^{\epsilon}_t\}$ is a regular
  perturbation of $\{{\mathcal{P}}^0_t\}$.\label{pro2}
\end{lemma}

\begin{proof}
  Pick any two states $z^1:=(s^0,s^1)$ and $z^2:=(s^1,s^2)$ with $z^1
  \neq z^2$. We have that $P^{\epsilon}_{z^1 z^2} > 0$ if and only if
  there is some $i\in V$ such that $s_{-i}^1 = s_{-i}^2$ and one of
  the following occurs: $s^2_i\in \FF_i(a_i^1)\setminus\{s^0_i,
  s^1_i\}$, $s^2_i=s^1_i$ or $s^2_i=s^0_i$. In particular, the
  following holds:
  \begin{align*}
    P^{\epsilon}_{z^1 z^2} =
    \begin{cases}
      \eta_1, \quad s^2_i\in \FF_i(a_i^1)\setminus\{s^0_i, s^1_i\},\\
      \eta_2, \quad s^2_i=s^1_i,\\
      \eta_3, \quad s^2_i=s^0_i,
    \end{cases}
  \end{align*}
  where
  \begin{align*}
    \eta_1:=\frac{\epsilon^{m_i}}{N|\FF_i(a_i^1)\setminus\{s^0_i, s^1_i\}|},\quad \eta_2:=\frac{1-\epsilon^{m_i}}{N(1+\epsilon^{\rho_i(s^0,s^1)})},\quad
    \eta_3:=\frac{(1-\epsilon^{m_i})\times\epsilon^{\rho_i(s^0,s^1)}}{N(1+\epsilon^{\rho_i(s^0,s^1)})}.
  \end{align*}

  Observe that $ 0<\lim_{\epsilon\rightarrow
    0^+}\frac{\eta_1}{\epsilon^{m_i}}<+\infty.$ Multiplying the
  numerator and denominator of $\eta_2$ by $\epsilon^{\Psi_i(s^1,
    s^0)-(u_i(s^1)-\Delta_i(s^1,s^0))}$,  we obtain
\begin{align*}
  \eta_2 = \frac{1-\epsilon^{m_i}}{N} \times
  \frac{\epsilon^{\Psi_i(s^0,
      s^1)-(u_i(s^1)-\Delta_i(s^1,s^0))}}{\eta_2'},
\end{align*}
where $\eta_2' := \epsilon^{\Psi_i(s^0,
  s^1)-(u_i(s^1)-\Delta_i(s^1,s^0))}+\epsilon^{\Psi_i(s^0,
  s^1)-(u_i(s^0)-\Delta_i(s^0,s^1))}$. Use \begin{align*} \lim_{\epsilon\rightarrow0^+}\epsilon^{x}
  = \begin{cases}
    1, & x = 0,\\
    0, & x > 0,
\end{cases}
\end{align*} and we have
\begin{align*}
  \lim_{\epsilon\rightarrow 0^+} \frac{\eta_2}{\epsilon^{\Psi_i(s^0, s^1)-(u_i(s^1)-\Delta_i(s^1,s^0))}}
  =\begin{cases}
    \frac{1}{N}, & u_i(s^0)-\Delta_i(s^0,s^1)\neq u_i(s^1)-\Delta_i(s^1,s^0),\\
    \frac{1}{2N}, & {\rm otherwise}.
  \end{cases}
\end{align*}
Similarly, it holds that
\begin{align*}
  \lim_{\epsilon\rightarrow 0^+} \frac{\eta_3}{\epsilon^{\Psi_i(s^0,
      s^1)-(u_i(s^0)-\Delta_i(s^0,s^1))}}\in\{\frac{1}{2N},
  \frac{1}{N}\}.
\end{align*}
Hence, the resistance of the feasible transition $z^1\rightarrow z^2$,
with $z^1\neq z^2$ and sensor $i$ as the unilateral deviator, can be
described as follows:
\begin{align*}
  \chi(z^1\rightarrow z^2) =
  \begin{cases}
    m_i, \quad s^2_i\in \FF_i(a^1)\setminus\{s^0_i, s^1_i\},\\
    \Psi_i(s^0, s^1)-(u_i(s^1)-\Delta_i(s^1,s^0)), & s^2_i=s^1_i,\\
    \Psi_i(s^0, s^1)-(u_i(s^0)-\Delta_i(s^0,s^1)), & s^2_i=s^0_i.
\end{cases}
\end{align*}


Then (A3) in Section~\ref{sec:resistancetrees} holds. It is
straightforward to verify that (A2) in
Section~\ref{sec:resistancetrees} holds. We are now in a position to
verify (A1). Since $\subscr{\GG}{loc}$ is undirected and connected,
and multiple sensors can stay in the same position, then $\diamond a^0
= {\mathcal{Q}}^N$ for any $a^0\in {\mathcal{Q}}$. Since sensor $i$
can choose any camera control vector from $\mathcal{C}$ at each time,
then $\diamond s^0 = {\mathcal{A}}$ for any
$s^0\in{\mathcal{A}}$. This implies that $\diamond z^0 =
{\mathcal{B}}'$ for any $z^0 \in {\mathcal{B}}'$, and thus the Markov
chain $\{{\mathcal{P}}^{\epsilon}_t\}$ is irreducible on the space
${\mathcal{B}}'$.

It is easy to see that any state in $\diag{\Ac}$ has period $1$. Pick
any $(s^0, s^1)\in{\mathcal{B}}'\setminus\diag{\Ac}$. Since
$\subscr{\mathcal{G}}{loc}$ is undirected, then $s_i^0\in
{\FF}_i(a_i^1)$ if and only if $s_i^1\in {\FF}_i(a_i^0)$. Hence, the
following two paths are both feasible:
\begin{align*}
  &(s^0, s^1)\rightarrow (s^1, s^0) \rightarrow (s^0, s^1)\nnum\\
  &(s^0, s^1)\rightarrow (s^1, s^1) \rightarrow (s^1, s^0) \rightarrow
  (s^0, s^1).
\end{align*}
Hence, the period of the state $(s^0,s^1)$ is $1$. This proves
aperiodicity of $\{{\mathcal{P}}^{\epsilon}_t\}$. Since
$\{{\mathcal{P}}^{\epsilon}_t\}$ is irreducible and aperiodic, then
(A1) holds.
\end{proof}

A direct result of Lemma~\ref{pro2} is that for each $\epsilon>0$,
there exists a unique stationary distribution of
$\{{\mathcal{P}}^{\epsilon}_t\}$, say ${\mu}(\epsilon)$. From the
proof of Lemma~\ref{pro2}, we can see that the resistance of an
experiment is $m_i$ if sensor $i$ is the unilateral deviator. We now
proceed to utilize Theorem~\ref{the4} to characterize
$\lim_{\epsilon\rightarrow0^+}{\mu}(\epsilon)$.

\begin{proposition}
  Consider the regular perturbed Markov process
  $\{{\mathcal{P}}^{\epsilon}_t\}$. Then
  $\lim_{\epsilon\rightarrow0^+}{\mu}(\epsilon)$ exists and the
  limiting distribution ${\mu}(0)$ is a stationary distribution of
  $\{{\mathcal{P}}^0_t\}$. Furthermore, the stochastically stable
  states (i.e., the support of ${\mu}(0)$) are contained in the set
  $\diag{S^*}$.\label{pro5}
\end{proposition}

\begin{proof}
  The unperturbed Markov chain corresponds to the $\homoasyn$
  Algorithm with $\epsilon = 0$. Hence, the recurrent communication
  classes of the unperturbed Markov chain are contained in the set
  $\diag{\Ac}$. We will construct resistance trees over vertices in
  the set $\diag{\Ac}$. Denote $T_{\min}$ by the minimum resistance
  tree. The remainder of the proof is divided into the following four
  claims.

  \begin{claim} $\chi((s^0,s^0)\Rightarrow (s^1,s^1)) = m_i +
  \Psi_i(s^1, s^0)-(u_i(s^1)-\Delta_i(s^1,s^0))$ where $s^0\neq s^1$
  and the transition $s^0\rightarrow s^1$ is feasible with sensor $i$
  as the unilateral deviator.\end{claim}

  \begin{proof} One feasible path for
  $(s^0,s^0)\Rightarrow(s^1,s^1)$ is ${\LL}:=(s^0,s^0)\rightarrow
  (s^0,s^1)\rightarrow (s^1,s^1)$ where sensor $i$ experiments in the
  first transition and does not experiment in the second
  one. The total resistance of the path
  ${\LL}$ is $m_i + \Psi_i(s^1, s^0) - (u_i(s^1)-\Delta_i(s^1,s^0))$
  which is at most $m_i + m^*$.

  Denote by ${\LL}'$ the path with minimum resistance among all the
  feasible paths for $(s^0,s^0)\Rightarrow(s^1,s^1)$. Assume that the
  first transition in ${\LL}'$ is $(s^0,s^0)\rightarrow (s^0,s^2)$
  where node $j$ experiments and $s^2 \neq s^1$. Observe that the
  resistance of $(s^0,s^0)\rightarrow (s^0,s^2)$ is $m_j$. No matter
  whether $j$ is equal to $i$ or not, the path ${\LL}'$ must include
  at least one more experiment to introduce $s_i^1$. Hence the total
  resistance of the path ${\LL}'$ is at least $m_i+m_j$. Since $m_i +
  m_j > m_i+2m^*$, then the path ${\LL}'$ has a strictly larger
  resistance than the path ${\LL}$. To avoid a contradiction, the path
  ${\LL}'$ must start from the transition
  $(s^0,s^0)\rightarrow(s^0,s^1)$. Similarly, the sequent transition
  (which is also the last one) in the path ${\LL}'$ must be
  $(s^0,s^1)\rightarrow(s^1,s^1)$ and thus ${\LL}' = {\LL}$. Hence,
  the resistance of the transition $(s^0,s^0)\Rightarrow (s^1,s^1)$ is
  the total resistance of the path ${\LL}$; i.e., $m_i + \Psi_i(s^1,
  s^0)-(u_i(s^1)-\Delta_i(s^1,s^0))$.\end{proof}

  \begin{claim} All the edges $((s,s), (s',s'))$ in $T_{\min}$
  must consist of only one deviator; i.e., $s_i \neq s_i'$ and $s_{-i}
  = s_{-i}'$ for some $i\in V$.\label{claim9}\end{claim}

  \begin{proof} Assume that $(s,s)\Rightarrow (s',s')$ has
  at least two deviators. Suppose the path $\hat{\LL}$ has the minimum
  resistance among all the paths from $(s,s)$ to $(s',s')$. Then, $\ell\geq2$ experiments are
  carried out along $\hat{\LL}$. Denote $i_k$ by the unilateral deviator in the $k$-th
  experiment $s^{k-1}\rightarrow s^k$ where $1\leq k \leq \ell$,
  $s^0=s$ and $s^{\ell}=s'$. Then the resistance of $\hat{\LL}$ is at least
  $\sum_{k=1}^{\ell}m_{i_k}$; i.e., $\chi((s^0, s^0)\Rightarrow(s',
  s'))\geq \sum_{k=1}^{\ell}m_{i_k}$.

  Let us consider the following path on $T_{\min}$:
  \begin{align*} \bar{\LL} := (s^0,s^0)\Rightarrow (s^1,s^1)
    \Rightarrow \cdots\Rightarrow (s^{\ell}, s^{\ell}).
  \end{align*} From Claim 1, we know that the total resistance of the
  path $\bar{\LL}$ is at most $\sum_{k=1}^{\ell}m_{i_k} + \ell m^*$.

  A new tree $T'$ can be obtained by adding the edges of $\bar{\LL}$
  into $T_{\min}$ and removing the redundant edges. The removed
  resistance is $strictly$ greater than $\sum_{k=1}^{\ell}m_{i_k} +
  2(\ell-1)m^*$ where $\sum_{k=1}^{\ell}m_{i_k}$ is the lower bound on the
  resistance on the edge from $(s^0, s^0)$ to $(s^{\ell}, s^{\ell})$, and $2(\ell-1)m^*$ is
  the strictly lower bound on the total resistances of leaving $(s^k,
  s^k)$ for $k = 1, \cdots, \ell-1$. The adding resistance is the total
  resistance of $\bar{\LL}$ which is at most
  $\sum_{k=1}^{\ell}m_{i_k} + \ell m^*$. Since $\ell\geq2$, we have that
  $2(\ell-1)m^* \geq \ell m^*$ and thus $T'$ has a strictly lower resistance
  than $T_{\min}$. This contradicts the fact that $T_{\min}$ is a
  minimum resistance tree.\end{proof}

  \begin{claim} Given any edge $((s,s), (s',s'))$ in $T_{\min}$,
  denote by $i$ the unilateral deviator between $s$ and $s'$. Then the
  transition $s_i\rightarrow s_i'$ is feasible.\end{claim}

  \begin{proof} Assume that the transition $s_i\rightarrow s_i'$ is
  infeasible. Suppose the path $\check{\LL}$ has the minimum
  resistance among all the paths from $(s,s)$ to $(s',s')$. Then,
  there are $\ell\geq2$ experiments in $\check{\LL}$. The remainder of
  the proof is similar to that of Claim 9.\end{proof}

  \begin{claim} Let $h_v$ be the root of $T_{\min}$. Then,
  $h_v\in\diag{S^*}$.\end{claim}

  \begin{proof} Assume that $h_v=(s^0,s^0)\notin \diag{S^*}$. Pick
  any $(s^*, s^*)\in\diag{S^*}$. By Claim 9 and 10, we have that there
  is a path from $(s^*,s^*)$ to $(s^0,s^0)$ in the tree $T_{\min}$ as
  follows:
  \begin{align*} \tilde{\mathcal{L}}:=(s^{\ell},s^{\ell})\Rightarrow
    (s^{\ell-1},s^{\ell-1})\Rightarrow\cdots \Rightarrow
    (s^1,s^1)\Rightarrow(s^0,s^0)
\end{align*}
for some $\ell\geq1$. Here, $s^*=s^{\ell}$, there is only one deviator, say
$i_k$, from $s^k$ to $s^{k-1}$, and the transition
$s^k \rightarrow s^{k-1}$ is feasible for $k = \ell, \dots, 1$.

Since the transition $s^k\rightarrow s^{k+1}$ is also feasible for $k = 0,\dots, \ell-1$, we obtain the reverse path
  $\tilde{\LL}'$ of $\tilde{\LL}$ as follows:
  \begin{align*}
    \tilde{\LL}' := (s^0,s^0)\Rightarrow (s^{1},s^{1})\Rightarrow
    \cdots \Rightarrow (s^{\ell-1},s^{\ell-1})
    \Rightarrow(s^{\ell}, s^{\ell}).
  \end{align*}

  By Claim 8, the total resistance of the path
  $\tilde{\LL}$ is
  \begin{align*} \chi({\tilde{\LL}}) &= \sum_{k=1}^{\ell}m_{i_k}
    +\sum_{k=1}^{\ell}\{\Psi_{i_k}(s^k, s^{k-1})-(u_{i_k}(s^{k-1})-\Delta_{i_k}(s^{k-1},s^k))\},
  \end{align*}
  and the total resistance of the path $\tilde{\LL}'$ is
  \begin{align*}
    \chi(\tilde{\LL}')&=\sum_{k=1}^{\ell}m_{i_k}
    +\sum_{k=1}^{\ell}\Psi_{i_k}(s^{k-1},s^k)-(u_{i_k}(s^k)-\Delta_{i_k}(s^k, s^{k-1})).
  \end{align*}

  Denote $\Lambda_1' := ({\mathcal{D}}(a_{i_k}^k, r_{i_k}^k)\backslash{\mathcal{D}}(a_{i_{k-1}}^{k-1},
  r_{i_{k-1}}^{k-1}))\cap\QQ$ and $\Lambda_2' := ({\mathcal{D}}(a_{i_{k-1}}^{k-1},
  r_{i_{k-1}}^{k-1})\backslash{\mathcal{D}}(a_{i_k}^k, r_{i_k}^k))\cap\QQ$. Observe that
  \begin{align*}
    &U_g(s^k)-U_g(s^{k-1})\nnum\\
    &= u_{i_k}(s^k) - u_{i_k}(s^{k-1}) - \sum_{q\in \Lambda_1'}W_q(\frac{n_q(s^{k-1})}{n_q(s^{k-1})}-\frac{n_q(s^{k-1})}{n_q(s^k)})
    + \sum_{q\in \Lambda_2'}W_q(\frac{n_q(s^k)}{n_q(s^k)}-\frac{n_q(s^k)}{n_q(s^{k-1})}) \nnum\\
    & = (u_{i_k}(s^k) - \Delta_{i_k}(s^k, s^{k-1})) - (u_{i_k}(s^{k-1}) - \Delta_{i_k}(s^{k-1}, s^k)).
\end{align*}


We now construct a new tree $T'$ with the root $(s^*,s^*)$ by adding
the edges of $\tilde{\LL}'$ to the tree $T_{\min}$ and removing the
redundant edges $\tilde{\LL}$. Since $\Psi_{i_k}(s^{k-1}, s^k) =
\Psi_{i_k}(s^k, s^{k-1})$, the difference in the total resistances
across the trees $\chi(T')$ and $\chi(T_{\min})$ is given by
\begin{align*}
  &\chi(T')-\chi(T_{\min}) = \chi(\tilde{\LL}')-\chi(\tilde{\LL})\nnum\\
  &=\sum_{k=1}^{\ell}-(u_{i_k}(s^{k-1})-\Delta_{i_k}(s^{k-1},s^k))
  -\sum_{k=1}^{\ell}-(u_{i_k}(s^k)-\Delta_{i_k}(s^k,s^{k-1}))\nnum\\
  &=\sum_{k=1}^{\ell}(U_g(s^k)-U_g(s^{k-1}))=U_g(s^0)-U_g(s^*)<0.
\end{align*}
This contradicts that $T_{\min}$ is a minimum resistance
tree.\end{proof}

It follows from Claim 4 that
the state $h_v\in\diag{S^*}$ has minimum stochastic potential. Then
Proposition~\ref{pro5} is a direct result of Theorem~\ref{the4}.
\end{proof}


We are now ready to show Theorem~\ref{the2}.

\textbf{Proof of Theorem~\ref{the8}:}

\begin{claim} Condition (B2) in Theorem~\ref{the5} holds.\end{claim}

\begin{proof} The proof is analogous to Claim 5.\end{proof}

\begin{claim} Condition (B3) in Theorem~\ref{the5} holds.\end{claim}

\begin{proof}
Denote by ${P}^{\epsilon(t)}$ the transition matrix of
  $\{{\mathcal{P}}_t\}$. Consider the feasible transition
  $z^1\rightarrow z^2$ with unilateral deviator $i$. The corresponding
  probability is given by
\begin{align*}
P^{\epsilon(t)}_{z^1 z^2} =
\begin{cases}
  \eta_1, \quad s^2_i\in \FF_i(a_i^1)\setminus\{s^0_i, s^1_i\},\\
  \eta_2, \quad s^2_i=s^1_i,\\
  \eta_3, \quad s^2_i=s^0_i,
\end{cases}
\end{align*}
where
\begin{align*}
  \eta_1:=\frac{\epsilon(t)^{m_i}}{N|\FF_i(a_i^1)\setminus\{s^0_i, s^1_i\}|},\quad
  \eta_2:=\frac{1-\epsilon(t)^{m_i}}{N(1+\epsilon(t)^{\rho_i(s^0,s^1)})},\quad
  \eta_3:=\frac{(1-\epsilon(t)^{m_i})\times\epsilon(t)^{\rho_i(s^0,s^1)}}{N(1+\epsilon(t)^{\rho_i(s^0,s^1)})}.
\end{align*}

The remainder is analogous to Claim 6.\end{proof}

\begin{claim} Condition (B1) in Theorem~\ref{the5} holds.\end{claim}

\begin{proof}
Observe that $|\FF_i(a_i^1)|\leq5|{\mathcal{C}}|$. Since $\epsilon(t)$ is strictly decreasing, there
is $t_0\geq 1$ such that $t_0$ is the first time when
$1-\epsilon(t)^{m_i}\geq \epsilon(t)^{m_i}$.


Observe that for all $t\geq1$, it holds that
\begin{align*}
  \eta_1 \geq \frac{\epsilon(t)^{m_i}}{N(5|{\mathcal{C}}|-1)}\geq
  \frac{\epsilon(t)^{m_i+m^*}}{N(5|{\mathcal{C}}|-1)}.
\end{align*}

Denote $b := u_i(s^1)-\Delta_i(s^1,s^0)$ and $a :=
u_i(s^0)-\Delta_i(s^0,s^1)$. Then $\rho_i(s^0, s^1) = b-a$. Since $b-a
\leq m^*$, then for $t\geq t_0$ it holds that
\begin{align*} & \eta_2 =
  \frac{1-\epsilon(t)^{m_i}}{N(1+\epsilon(t)^{b-a})}
  =\frac{(1-\epsilon(t)^{m_i})\epsilon(t)^{\max\{a,b\}-b}}{N(\epsilon(t)^{\max\{a,b\}-b}+\epsilon(t)^{\max\{a,b\}-a})}\nnum\\
  &\geq \frac{\epsilon(t)^{m_i}\epsilon(t)^{\max\{a,b\}-b}}{2N}\geq
  \frac{\epsilon(t)^{m_i+m^*}}{N(5|{\mathcal{C}}|-1)}.
\end{align*}
Similarly, for $t\geq t_0$, it holds that
\begin{align*} \eta_3 =
  \frac{(1-\epsilon(t)^{m_i})\epsilon(t)^{\max\{a,b\}-a}}{N(\epsilon(t)^{\max\{a,b\}-b}+\epsilon(t)^{\max\{a,b\}-a})}
  \geq \frac{\epsilon(t)^{m_i+m^*}}{N(5|{\mathcal{C}}|-1)}.
\end{align*}
Since $m_i\in(2m^*, Km^*]$ for all $i \in V$ and $Km^* > 1$, then for any
feasible transition $z^1\rightarrow z^2$ with $z^1\neq z^2$, it holds
that:
\begin{align*}
  P^{\epsilon(t)}_{z^1 z^2}\geq
  \frac{\epsilon(t)^{(K+1)m^*}}{N(5|{\mathcal{C}}|-1)}
\end{align*}
for all $t\geq t_0$. Furthermore, for all $t\geq t_0$ and all $z^1\in
\diag{\Ac}$, we have that:
\begin{align*}P_{z^1z^1}^{\epsilon(t)} =
  1-\frac{1}{N}\sum_{i=1}^N\epsilon(t)^{m_i} =
  \frac{1}{N}\sum_{i=1}^N(1-\epsilon(t)^{m_i})\geq
  \frac{1}{N}\sum_{i=1}^N\epsilon(t)^{m_i}\geq
  \frac{\epsilon(t)^{(K+1)m^*}}{N(5|{\mathcal{C}}|-1)}.
\end{align*}


Choose $k_{i} := (D+1){i}$ and let $i_0$ be the smallest integer such
that $(D+1)i_0\geq t_0$. Similar to~\eqref{e8}, we can derive the
following property
\begin{align*}
  \sum_{\ell=0}^{+\infty}(1-\lambda(P(k_{\ell}, k_{\ell+1})))
  \geq\frac{|\mathcal{B}|}{(N(5|{\mathcal{C}}|-1))^{(D+1)(K+1)m^*}}\sum_{i=i_0}^{+\infty}\frac{1}{(D+1)i}
  = +\infty.
\end{align*}
Hence, the weak ergodicity of $\{{\mathcal{P}}_t\}$ follows from
Theorem~\ref{the7}.\end{proof}

All the conditions in Theorem~\ref{the5} hold. Thus it follows from Theorem~\ref{the5} that
the limiting distribution is $\mu^* =
\lim_{t\rightarrow+\infty}\mu^t$. Note that
$\lim_{t\rightarrow+\infty}\mu^t =
\lim_{t\rightarrow+\infty}\mu(\epsilon(t)) = \mu(0)$ and
Proposition~\ref{pro5} shows that the support of $\mu(0)$ is contained in the set
$\diag{S^*}$. Hence, the support of $\mu^*$ is contained in the set
$\diag{S^*}$, implying that
$\lim_{t\rightarrow+\infty}{\mathbb{P}}(z(t)\in\diag{S^*}) = 1$. It
completes the proof.

\section{Conclusions}\label{sec:conclusion}

We have formulated a coverage optimization problem as a constrained
potential game. We have proposed two payoff-based distributed learning algorithms
for this coverage game and shown that these algorithms converge in
probability to the set of constrained NEs and the set of global optima
of certain coverage performance metric, respectively.

\section{Appendix}\label{sec:appendix}

For the sake of a self-contained exposition, we include here some
background in Markov chains~\cite{DI-RM:76} and the
Theory of Resistance Trees~\cite{HPY:93}.

\subsection{Background in Markov chains}\label{sec:Markovchain}

A \emph{discrete-time Markov chain} is a discrete-time stochastic
process on a finite (or countable) state space and satisfies the
Markov property (i.e., the future state depends on its present state,
but not the past states). A discrete-time Markov chain is said to be
\emph{time-homogeneous} if the probability of going from one state to
another is independent of the time when the step is taken. Otherwise,
the Markov chain is said to be \emph{time-inhomogeneous}.

Since time-inhomogeneous Markov chains include time-homogeneous ones
as special cases, we will restrict our attention to the former in the
remainder of this section. The evolution of a time-inhomogeneous
Markov chain $\{{\mathcal{P}}_t\}$ can described by the transition
matrix $P(t)$ which gives the probability of traversing from one state
to another at each time $t$.

Consider a Markov chain $\{{\mathcal{P}}_t\}$ with time-dependent
transition matrix $P(t)$ on a finite state space $X$. Denote by $P(m,n)
:= \prod_{t=m}^{n-1}P(t)$, $0 \leq m < n$.

\begin{definition}[Strong ergodicity~\cite{DI-RM:76}] The Markov chain
  $\{{\mathcal{P}}_t\}$ is strongly ergodic if there exists a
  stochastic vector $\mu^*$ such that for any distribution $\mu$ on
  $X$ and any $m\in {\mathbb{Z}}_+$, it holds that $\lim_{k\rightarrow
    +\infty}\mu^TP(m,k) = (\mu^*)^T$.\label{def6}
\end{definition}

Strong ergodicity of $\{{\mathcal{P}}_t\}$ is equivalent to
$\{{\mathcal{P}}_t\}$ being convergent in distribution and will be
employed to characterize the long-run properties of our learning
algorithm. The investigation of conditions under which strong
ergodicity holds is aided by the introduction of the coefficient of
ergodicity and weak ergodicity defined next.

\begin{definition}[Coefficient of ergodicity~\cite{DI-RM:76}] For any
  $n\times n$ stochastic matrix $P$, its coefficient of ergodicity is
  defined as $\lambda(P) := 1-\min_{1\leq i,j\leq
    n}\sum_{k=1}^n\min(P_{ik,}P_{jk}).$
\label{def7}
\end{definition}

\begin{definition}[Weak ergodicity~\cite{DI-RM:76}] The Markov chain
  $\{{\mathcal{P}}_t\}$ is weakly ergodic if $\forall x,y,z\in X$,
  $\forall m\in {\mathbb{Z}}_+$, it holds that $\lim_{k \rightarrow
    +\infty}(P_{xz}(m,k) - P_{yz}(m,k)) = 0$.\label{def5}
\end{definition}

Weak ergodicity merely implies that $\{{\mathcal{P}}_t\}$
asymptotically forgets its initial state, but does not guarantee
convergence. For a time-homogeneous Markov chain, there is no
distinction between weak ergodicity and strong ergodicity. The
following theorem provides the sufficient and necessary condition for
$\{{\mathcal{P}}_t\}$ to be weakly ergodic.

\begin{theorem}[\cite{DI-RM:76}] The Markov chain $\{{\mathcal{P}}_t\}$
  is weakly ergodic if and only if there is a strictly increasing
  sequence of positive numbers $k_i$, $i\in {\mathbb{Z}}_+$ such that
  $\sum_{i=0}^{+\infty}(1-\lambda(P(k_i,k_{i+1})) = +\infty.$\label{the7}
\end{theorem}


We are now ready to present the sufficient conditions for strong
ergodicity of the Markov chain $\{{\mathcal{P}}_t\}$.

\begin{theorem}[\cite{DI-RM:76}] A Markov chain $\{{\mathcal{P}}_t\}$
  is strongly ergodic if the following conditions hold:

  (B1) The Markov chain $\{{\mathcal{P}}_t\}$ is weakly ergodic.

  (B2) For each $t$, there exists a stochastic vector $\mu^t$ on $X$
  such that $\mu^t$ is the left eigenvector of the transition matrix
  $P(t)$ with eigenvalue $1$.

  (B3) The eigenvectors $\mu^t$ in (B2) satisfy
  $\sum_{t=0}^{+\infty}\sum_{z\in X}|\mu^t_z-\mu^{t+1}_z|<+\infty$.

  Moreover, if $\mu^* = \lim_{t\rightarrow +\infty}\mu^t$, then $\mu^*$
  is the vector in Definition~\ref{def6}.\label{the5}
\end{theorem}

\subsection{Background in the Theory of Resistance
  Trees}\label{sec:resistancetrees}

Let $P^0$ be the transition matrix of the time-homogeneous Markov
chain $\{{\mathcal{P}}^0_t\}$ on a finite state space $X$. And let
$P^{\epsilon}$ be the transition matrix of a \emph{perturbed Markov
  chain}, say $\{{\mathcal{P}}^{\epsilon}_t\}$. With probability
$1-\epsilon$, the process $\{{\mathcal{P}}^{\epsilon}_t\}$ evolves
according to $P^0$, while with probability $\epsilon$, the transitions
do not follow $P^0$.

A family of stochastic processes $\{{\mathcal{P}}^{\epsilon}_t\}$ is
called a \emph{regular perturbation} of $\{{\mathcal{P}}^0_t\}$ if the
following holds $\forall x,y\in X$: (A1) For some $\varsigma>0$, the Markov chain
$\{{\mathcal{P}}^{\epsilon}_t\}$ is irreducible and aperiodic for all
$\epsilon\in(0,\varsigma]$.

(A2) $\lim_{\epsilon\rightarrow0^+} P_{xy}^{\epsilon} = P_{xy}^0$.

(A3) If $P_{xy}^{\epsilon}>0$ for some $\epsilon$, then there exists a
real number $\chi(x\rightarrow y)\geq0$ such that
$\lim_{\epsilon\rightarrow0^+}P_{xy}^{\epsilon}/\epsilon^{\chi(x\rightarrow
  y)}\in(0,+\infty)$.

In (A3), $\chi(x\rightarrow y)$ is called
the resistance of the transition from $x$ to $y$.

Let $H_1,H_2,\cdots,H_J$ be the recurrent communication classes of the
Markov chain $\{{\mathcal{P}}^0_t\}$. Note that within each class
$H_\ell$, there is a path of zero resistance from every state to every
other. Given any two distinct recurrence classes $H_\ell$ and $H_k$,
consider all paths which start from $H_\ell$ and end at $H_k$. Denote
$\chi_{\ell k}$ by the least resistance among all such paths.

Now define a complete directed graph $\mathcal{G}$ where there is one
vertex $\ell$ for each recurrent class $H_\ell$, and the resistance on
the edge $(\ell,k)$ is $\chi_{\ell k}$. An $\ell$-\emph{tree} on
$\mathcal{G}$ is a spanning tree such that from every vertex $k\neq
\ell$, there is a unique path from $k$ to $\ell$. Denote by $G(\ell)$
the set of all $\ell$-\emph{trees} on $\mathcal{G}$. The resistance of
an $\ell$-\emph{tree} is the sum of the resistances of its edges. The
\emph{stochastic potential} of the recurrent class $H_\ell$ is the
least resistance among all $\ell$-\emph{trees} in $G(\ell)$.

\begin{theorem}[\cite{HPY:93}] Let $\{{\mathcal{P}}^{\epsilon}_t\}$ be
  a regular perturbation of $\{{\mathcal{P}}^0_t\}$, and for each
  $\epsilon>0$, let $\mu(\epsilon)$ be the unique stationary
  distribution of $\{{\mathcal{P}}^{\epsilon}_t\}$. Then
  $\lim_{\epsilon\rightarrow0^+}\mu(\epsilon)$ exists and the limiting
  distribution $\mu(0)$ is a stationary distribution of
  $\{{\mathcal{P}}^0_t\}$. The stochastically stable states (i.e., the
  support of $\mu(0)$) are precisely those states contained in the
  recurrence classes with minimum stochastic potential.\label{the4}
\end{theorem}

\end{document}